\documentclass[10pt]{article}  	% use "amsart" instead of "article" for AMSLaTeX format
\usepackage{geometry}                		% See geometry.pdf to learn the layout options. There are lots.
\usepackage{graphicx}				% Use pdf, png, jpg, or epsÂ§ with pdflatex; use eps in DVI mode
\usepackage{amssymb}
\usepackage{mathtools}
\usepackage{graphicx}
\usepackage{amsmath}
\usepackage{amsthm}
\usepackage{algorithm}
\usepackage{algpseudocode}
\usepackage{tikz}
\usepackage{parskip}
\usepackage{hyperref}
\usepackage[normalem]{ulem} 
\usepackage{authblk}
\usepackage{ulem}
\usepackage{tabularx}

\usepackage{subcaption}

\graphicspath{ {\string~/Desktop/} }

\newtheorem{theorem}{Theorem}[section]
\newtheorem{lemma}[theorem]{Lemma}

\theoremstyle{definition}
\newtheorem{definition}[theorem]{Definition}

\theoremstyle{definition}
\newtheorem{example}[theorem]{Example}
\newtheorem{remark}[theorem]{Remark}

\DeclareMathOperator*{\argmin}{argmin}

\newcommand{\R}{{\mathbb{R}}}
\newcommand{\C}{{\mathbb{C}}}

\newcommand{\F}{{\mathbb{F}}}

\newcommand{\Fb}{ \boldsymbol{{\mathbb{F}}}}
\newcommand{\Bc}{\boldsymbol{{\mathcal{B}}}}
\newcommand{\Mc}{\boldsymbol{{\mathcal{M}}}}
\newcommand{\Lc}{ \boldsymbol{{\mathcal{L}}}}
\newcommand{\Xc}{ \boldsymbol{{\mathcal{X}}}}
\newcommand{\Tc}{\boldsymbol{{\mathcal{T}}}}
\newcommand{\Uc}{\boldsymbol{{\mathcal{U}}}}
\newcommand{\Sc}{\boldsymbol{{\mathcal{S}}}}
\newcommand{\Rc}{\boldsymbol{{\mathcal{R}}}}
\newcommand{\Zc}{\boldsymbol{{\mathcal{Z}}}}
\newcommand{\Yc}{\boldsymbol{{\mathcal{Y}}}}

\newcommand{\Oc}{\boldsymbol{{\mathcal{O}}}}

\newcommand{\Vc}{\boldsymbol{{\mathcal{V}}}}

\newcommand{\diag}{{\rm{diag}}}
\newcommand{\rank}{{\rm{rank}}}

\DeclareMathOperator*{\vecc}{{\rm{vec}}}

\DeclareMathOperator*{\im}{{\rm{im}}}

\DeclareMathOperator*{\mat}{{\rm{mat}}}

\DeclareMathOperator*{\ten}{{\rm{ten}}}

\newcommand{\svd}{{\rm{svd}}}
\newcommand{\vb}{\mathbf{b}}
\newcommand{\vc}{\mathbf{c}}

\newcommand{\vr}{\mathbf{r}}

\newcommand{\vu}{\mathbf{u}}
\newcommand{\vv}{\mathbf{v}}

\newcommand{\vx}{\mathbf{x}}
\newcommand{\vy}{\mathbf{y}}
\newcommand{\vz}{\mathbf{z}}

\newcommand{\mA}{\mathbf{A}}
\newcommand{\mB}{\mathbf{B}}
\newcommand{\mC}{\mathbf{C}}

\newcommand{\mE}{\mathbf{E}}

\newcommand{\mG}{\mathbf{G}}
\newcommand{\mH}{\mathbf{H}}
\newcommand{\mI}{\mathbf{I}}
\newcommand{\mJ}{\mathbf{J}}

\newcommand{\mL}{\mathbf{L}}
\newcommand{\mM}{\mathbf{M}}

\newcommand{\mP}{\mathbf{P}}

\newcommand{\mR}{\mathbf{R}}
\newcommand{\mS}{\mathbf{S}}

\newcommand{\mU}{\mathbf{U}}
\newcommand{\mV}{\mathbf{V}}
\newcommand{\mW}{\mathbf{W}}
\newcommand{\mX}{\mathbf{X}}
\newcommand{\mY}{\mathbf{Y}}
\newcommand{\mZ}{\mathbf{Z}}

\theoremstyle{remark}

\title{The Generalized Matrix Separation Problem: Algorithms}

\author{Xuemei Chen\thanks{
                  Department of Mathematics and Statistics, University of North Carolina, Wilmington NC, USA. Email: chenxuemei@uncw.edu}
         \ and Owen Deen\thanks{
                  Department of Mathematics, University of Maryland, College Park MD, USA. Email: odeen@umd.edu }%\thanks{This work was
        }

\begin{document}

\maketitle

\begin{abstract}
When given a generalized matrix separation problem, which aims to recover a low rank matrix $L_0$ and a sparse matrix $S_0$ from $M_0=L_0+HS_0$, the work \cite{CW25} proposes a novel convex optimization problem whose objective function is the sum of the $\ell_1$-norm and nuclear norm. In this paper we detail the iterative algorithms and its associated computations for solving this convex optimization problem. We present various efficient implementation strategies, with attention to practical cases where $H$ is circulant, separable, or block structured. Notably, we propose a preconditioning technique that drastically improved the performance of our algorithms in terms of efficiency, accuracy, and robustness. While this paper serves as an illustrative algorithm implementation manual, we also provide theoretical guarantee for our preconditioning strategy. Numerical results demonstrate the effectiveness of the proposed approach.
\end{abstract}

\section{Introduction} \label{sec:intro}
%This paper addresses the problem of recovering structured signals from high-dimensional data through matrix decomposition. The goal is to separate a matrix into a low-rank component, capturing the underlying structure, and a sparse component, representing noise or outliers. Matrix separation has applications in image processing, signal analysis, and related fields \cite{CLMW11}.
 
 Given a known matrix $\mM_0 \in \R^{m \times n}$ and a known linear operator $\mH \in \mathbb{R}^{m\times p}$, the \emph{generalized matrix separation problem} consists of recovering a low-rank matrix $\mL_0 \in \mathbb{R}^{m \times n}$ and a sparse matrix $\mS_0 \in \R^{p\times n}$ such that
 \begin{equation}\label{equ:M0}
 \mM_0 = \mL_0 + \mH\mS_0.
 \end{equation} 
%Here $\mL_0$ captures the underlying low-dimensional structure of the data, while $\mS_0$ is sparse in the canonical basis but is observed only through the linear transformation $\mH$, so that $\mH\mS_0$ is not necessarily sparse.
 As shown in \eqref{equ:M0}, the sparse matrix component is filtered or masked by a known linear transformation $\mH \in \mathbb{R}^{m \times p}$, and the product $\mH S_0$ is not necessarily sparse.
  
The separation problem \eqref{equ:M0} was first formulated in \cite{CW25}, which  proposes the following convex optimization problem to recover $\mL_0$ and $\mS_0$:
\begin{equation}\label{equ:p}
(\hat \mS, \hat \mL) = \argmin_{\mS, \mL} \lambda \|\mS\|_1 + \|\mL\|_*, \quad \text{subject to } \mL + \mH\mS = \mM_0,
\end{equation}
where $\lambda>0$ is a regularization parameter, $\| \cdot \|_*$ is the nuclear norm, which is the convex envelope of the matrix rank on the unit spectral norm ball, and $\| \cdot \|_1$ is the vectorized $\ell_1$-norm, which is the convex envelope of sparsity on the unit $\ell_\infty$ ball. 
 
When $\mH$ is the identity matrix, the problem reduces to the standard robust PCA~\cite{LCM10, C11, CLMW11} formulation: 
\begin{equation}\label{eq:matrix_sep} 
(\hat \mS, \hat \mL)=\argmin_{\mS, \mL}\lambda\|\mS\|_1+\|\mL\|_*, \quad\text{subject to }\mL+\mS=\mM_0.
\end{equation}

There are many algorithms for solving the robust PCA problem \eqref{eq:matrix_sep}. The alternating direction method of multipliers (ADMM)~\cite{BS11} algorithm is a popular choice \cite{wang2013solving, YXYJ09, LCM10}. An alternating projection algorithm for a nonconvex formulation is presented in \cite{netrapalli2014non}, and an accelerated version that significantly improves accuracy is proposed in \cite{CCW19}. The work~\cite{dutta2019nonconvex} also belongs to alternating projection methods but treats the constraint $\mL+\mS=\mM_0$ as a separate projection step. The work in \cite{Ce20} proposes another method for a nonconvex formulation using a CUR approximation.

While closely related to Robust PCA, the generalized matrix separation problem considered in this work is more flexible. Robust PCA typically assumes a decomposition of the form $\mM_0=\mL+\mS$, where $\mL$ is low-rank and $S$ is sparse. In contrast, the formulation studied here allows for more general linear operators acting on the sparse component through $\mH$, leading to a broader class of matrix separation problems. Consequently, the algorithms developed in this work are designed specifically for this generalized formulation and differ from methods developed for the classical Robust PCA setting.

The matrix $\mH$ can be of any shape and may be rank-deficient. However, we generally require $\mH$ to behave well enough that it is injective on $\Omega(\mS_0)$, the set of all $p \times n$ matrices whose support lies within the support of $\mS_0$. See Definition~\ref{def:rinp} for details.
%We refer the reader to \cite{CW25} for more details.

The setup \eqref{equ:M0} and the convex optimization problem \eqref{equ:p} are new. In this paper, we address the iterative steps for solving \eqref{equ:p} using ADMM in detail. ADMM was introduced in \cite{glowinski1975approximation} and \cite{gabay1976dual}, and gained popularity due to its effectiveness in large-scale data analysis and its applications in signal processing and machine learning. Part of its appeal is that the algorithm’s updates lend themselves to parallel implementation. We refer the readers to \cite{BS11} for an excellent introduction.

\subsection*{Contribution and Organization}
The contributions of this paper are twofold: 
\begin{itemize}
\item We provide detailed and efficient algorithms for solving \eqref{equ:p}, along with implementable pseudocode. %Practical special cases, such as circulant $\mH$, are addressed. 
Moreover, Section~\ref{sec:tensor} examines the case where $\mM$ arises from a tensor, with the assumption that the filter $\mH$ is separable.
%which introduces memory and computational challenges. To address this, we focus our computations on separable filters that require fewer bits to store and allow for faster application to data.
\item We propose a preconditioning technique that significantly improves performance in terms of both efficiency and accuracy. The preconditioning does not change the row space or column space of $\mH$. Theorem~\ref{thm:prec} provides a theoretical guarantee for using the preconditioned filter under certain incoherence conditions on $\mH$, $\mL_0$, and $\mS_0$. Numerical experiments show that this technique is very crucial and significant for recovery guarantee. 
\end{itemize}

%In particular, for videos, circulant, and preconditioning... 

%We solve this optimization problem using the alternating direction method of multipliers (ADMM), a standard technique for convex problems with separable objective functions and linear constraints. 

%ADMM introduces an augmented Lagrangian and performs iterative updates by alternating between minimizing over each variable and updating the dual variable.

%We will solve \eqref{equ:p} using the Alternating Directions Method of Multipliers (ADMM), an iterative method that decomposes the optimization into subproblems, making it suitable for high-dimensional and noisy data.

The remainder of the paper is organized as follows: Section~\ref{sec:prelim} introduces the notation and necessary preliminaries. Section~\ref{sec:iter1d} details the mathematical framework for the matrix separation problem, followed by the preconditioning technique in Section~\ref{sec:prec} and its extension to video data in Section~\ref{sec:tensor}. Extensive numerical results are presented in Section~\ref{sec:numex}. Section~\ref{sec:pvn} uses four experiments to demonstrate the superior performance of the preconditioning technique. Section~\ref{sec:rs} explores the maximum rank and sparsity allowed with two different choices of $H$ and three sparsity models. Section~\ref{sec:vid} applies our algorithm to a simultaneous background removal and deblurring problem on real video data. We end with a discussion including future work.

\section{Notations \& Preliminaries} \label{sec:prelim}

We use bold lowercase letters such as $\vx$ and $\vy$ to represent vectors, bold uppercase letters such as $\mX$ and $\mY$ for matrices, and cursive script such as $\boldsymbol{\mathcal{T}}$ for tensors.

Let $\F$ be either $\R$ or $\C$. Given $p \geq 1$, the \emph{$\ell_p$-norm} of a vector $\vx = (x_1, x_2, \cdots, x_n) \in \F^{n}$ is defined as $\|\vx\|_p = \left( \sum_{i=1}^n |\vx_i|^p \right)^{1/p}$.

For a matrix $\mX = (x_{ij}) \in \F^{m \times n}$, the matrix norms are defined as follows:
\begin{itemize}
\item The \emph{Frobenius norm} of $\mX$ is $\|\mX\|_F = \sqrt{\sum_{i=1}^m \sum_{j=1}^n |x_{ij}|^2}$.
\item Both the $\ell_1$-norm and the infinity norm of $\mX$ are based on its vectorization. That is, 
 $\|\mX\|_1 = \sum_{i=1}^m \sum_{j=1}^n |x_{ij}|$, and $\|\mX\|_\infty=\max_{i,j}|x_{ij}|$.
\item The \emph{nuclear norm} of $\mX$ is $\|\mX\|_* = \sum_{i=1}^r \sigma_i$, where $\sigma_1, \ldots, \sigma_r$ are the nonzero singular values of $\mX$ and $r = \operatorname{rank}(\mX)$.
\item The \emph{spectral norm} of $\mX$ is denoted by $\|\mX\|$. It is the largest singular value of $\mX$.
\end{itemize}

The transpose of $\mX$ is denoted by $\mX^\top$, and its conjugate transpose by $\mX^*$. For a matrix $\mX = (x_{ij}) \in \mathbb{F}^{n \times n}$, the vector of diagonal elements is defined as $\diag(\mX) = [x_{11}, x_{22}, \ldots, x_{nn}]^\top$. Conversely, given a vector $\vv$, $\diag(\vv)$ is the diagonal matrix whose diagonal entries are given by $\vv$. We let $\mI_k$ denote the $k \times k$ identity matrix, and $\mJ_{m,n}$ denote the $m \times n$ matrix of all ones.

The \emph{Hadamard product} of two matrices $\mX = (x_{ij}) \in \mathbb{F}^{m \times n}$ and $\mY = (y_{ij}) \in \mathbb{F}^{m \times n}$ is denoted by $\mX \odot \mY$ and defined elementwise as $(\mX \odot \mY)_{ij} = x_{ij} y_{ij}$ for all $i$ and $j$. Likewise, we use $\oslash$ to denote pointwise division.

For a vector $\vx \in \mathbb{F}^{mn}$, we define $\im_{m,n}(\vx)$ as the column-wise reshaping of $\vx$ into an $m \times n$ matrix:
\[
\im_{m,n}(\vx) := \begin{bmatrix} x_1 & x_{m+1} & \cdots & x_{m(n-1)+1} \\ x_2 & x_{m+2} & \cdots & x_{m(n-1)+2} \\ \vdots & \vdots & & \vdots \\ x_m & x_{2m} & \cdots & x_{mn} \end{bmatrix}.
\]
We may write $\im(\vx)$ when the intended shape is clear from context.

The \emph{vectorization} of a matrix $\mX = (x_{ij}) \in \F^{m \times n}$ is the inverse of $\im(\cdot)$, and is defined as $$\vecc(\mX) = [x_{11}, x_{21}, \cdots, x_{(m-1), n}, x_{m,n}]^\top.$$

\begin{definition}[Soft Thresholding Operator] \label{def:3.2}
\normalfont
The \emph{soft thresholding operator} with parameter $\lambda$, denoted by $S_{\lambda}(\cdot)$, is a mapping from $\R^n$ to $\R^n$. For any $\vy \in \R^n$ and each index $i$, where $1 \leq i \leq n$, the operator is defined as
\[
(S_{\lambda}(\vy))_i = 
\begin{cases} 
      y_i - \lambda & \text{if } y_i >  \lambda, \\
      y_i + \lambda & \text{if } y_i < -\lambda, \\
      0 & \text{if } -\lambda \leq y_i \leq \lambda.
\end{cases}
\]
\end{definition}

It is well known that
\begin{equation}\label{equ:st}
S_\lambda(a)=\argmin_{b\in\R}\left(\lambda|b|+\frac{1}{2}(b-a)^2\right).
\end{equation}

When applies to a vector or a matrix, $S_\lambda(\cdot)$ applies entry-wise.
%For a matrix $\mM$, $S_\lambda(\mM)$ is to apply soft thresholding to each entry of the matrix $\mM$.
%We can apply Definition \ref{def:3.2} to a matrix by performing an entrywise thresholding. Let $M= (m_{ij}) \in \mathbb{R}^{m\times n}$, we find $S_{\lambda}(M)$ by applying the soft thresholding operator to each entry, namely $S_{\lambda}(m_{ij})$.

\begin{definition}[Singular value thresholding operator] \label{def:3.4}
\normalfont
Given a matrix $\mY \in \mathbb{R}^{m\times n}$, we denote $D_\lambda(\mY)$ as the \emph{singular value thresholding operator}, 
$$
D_\lambda(\mY)= \mU S_{\lambda}(\mathbf{\mathbf{\Sigma}})\mV^\top,\text{ where }\mY = \mU\mathbf{\mathbf{\Sigma}} \mV^\top \text{ is any singular value decomposition of }\mY.
$$
\end{definition}

Analogous to \eqref{equ:st}, the following result is well established.
\begin{theorem}[\cite{CCS10}] \label{thm:svs} For $\lambda > 0$ and $\mY \in \mathbb{R}^{m\times n}$, then
$$
D_{\lambda}(\mY) = \argmin_{\mX\in\R^{m\times n}}\left(\lambda\| \mX \|_* + \frac{1}{2}\| \mX - \mY \|^2_F\right).
$$

\end{theorem}

%\subsection{Properties of the Kronecker Product}

For $\mA \in \F^{m \times n}$, $\mB \in \mathbb{F}^{p \times q}$, the \emph{Kronecker product} $\mA \otimes \mB$ is an $mp \times nq$  matrix defined by
    \[
    \mA \otimes \mB = 
    \begin{bmatrix}
        a_{11}  B & a_{12}  B & \cdots & a_{1n} B \\
        a_{21}  B & a_{22} B & \cdots & a_{2n} B \\
        \vdots & \vdots &  & \vdots \\
        a_{m1}  B & a_{m2}  B & \cdots & a_{mn}  B \\
    \end{bmatrix}.
    \]
%The Kronecker product extends the dimensions of a matrix while preserving the information of the components. We will use the following properties to enhance computational efficiency for our iterative method. 

Let $\mX\in\F^{q\times n}$, then
\begin{equation}\label{equ:vec}
(\mA\otimes \mB)\vecc(\mX)=\vecc(\mB\mX\mA^\top).
\end{equation}

%Below are well-known properties of the Kronecker product.
%
%\begin{enumerate}    
%    \item Mixed-product property:
%    If there exists matrices $\mC,\mD$ that can multiply with $\mA$ and $\mB$ respectively, then
%    \[
%    (\mA \otimes \mB) \cdot (\mC \otimes \mD) = (\mA\mC) \otimes (\mB\mD).
%    \]
%    
%    
%    \item Inverse of a Kronecker product:
%    If $\mA$ and $\mB$ are invertible, then $(\mA \otimes \mB)^{-1} = \mA^{-1} \otimes \mB^{-1}$.
%    
%    \item Transpose of a Kronecker product:
%    \begin{align*}
%    \text{Transpose:  }(\mA \otimes \mB)^{\top} = \mA^{\top} \otimes \mB^{\top}. \\
%    \text{Conjugate Transpose: }(\mA \otimes \mB)^{*} = \mA^{*} \otimes \mB^{*}.
%    \end{align*}
%\end{enumerate}

The following lemma is a straightforward computation whose proof can be found in the Appendix.
\begin{lemma}\label{lem}
Let $\mA\in\mathbb{F}^{n\times n}$ and  $\mB\in\mathbb{F}^{m\times m}$ be diagonal matrices. Then
\begin{equation}\label{equ:lem}
\im_{m, n}(\diag(\mA\otimes \mB))=\diag(\mB)[\diag(\mA)]^\top
\end{equation}
\end{lemma}

\section{Iterative Steps for the Generalized Matrix Separation Problem} \label{sec:iter1d}

Given convex functions $f$ and $g$, matrices $\mA, \mB$, and vector $\vc$, ADMM seeks to solve the problem 
\begin{equation}\label{equ:admm}
\min_{\vx, \vz} \{f(\vx)+g(\vz)\}, \quad\text{subject to }\mA\vx+\mB\vz=\vc.
\end{equation}
With an appropriately chosen parameter $\rho$, the general steps of ADMM for solving \eqref{equ:admm} is given below  \cite[(3.5)--(3.7)]{BS11}:
\begin{align}\label{equ:admm1}
\vx^{k+1} &= \argmin_x \left(f(\vx)+\frac{\rho}{2}\|\mA\vx+\mB\vz^k-\vc+\vu^k\|_2^2\right)\\
\vz^{k+1} &= \argmin_z \left(g(\vz)+\frac{\rho}{2}\|\mA\vx^{k+1}+\mB\vz-\vc+\vu^k\|_2^2\right)\\\label{equ:admm3}
\vu^{k+1} &= \vu^k + \mA\vx^{k+1} + \mB\vz^{k+1}-\vc.
\end{align}

The convex optimization problem \eqref{equ:p} is well suited to be solved by ADMM since the objective function is a sum of two convex functions, allowing it to be reduced to solving two simpler subproblems. 
The ADMM steps for solving \eqref{equ:p} are:
%To solve \eqref{equ:p}, we can use the steps of ADMM from :
 \begin{align}\label{equ:L}
	 \mL^{k+1} &= \argmin_{\mL} \left(\| \mL\|_* + \rho/2\| \mL+\mH\mS^k- \mM+\mU^k\|_F^2\right)\\ \label{equ:S}
	\mS^{k+1} &= \argmin_{\mS} \left(\lambda \|\mS\|_1+\rho/2\| \mL^{k+1}+\mH\mS- \mM+\mU^k\|_F^2\right) \\ \label{equ:Y}
	\mU^{k+1} &= \mU^k + \mL^{k+1}+\mH\mS^{k+1} - \mM.
\end{align}
Note that the nuclear norm and $\ell_1$-norm terms are handled independently, allowing each subproblem to be solved efficiently.
%: the low-rank update uses singular value thresholding, and the sparse update uses soft-thresholding. 
The dual variable is then updated using the residual of the constraint $\mL + \mH\mS = \mM_0$. 

The solution to \eqref{equ:L} is exact and given by $\mL^{k+1} = D_{\rho^{-1}}(-\mH\mS^k + \mM - \mU^k)$ according to Theorem \ref{thm:svs}.

The problem \eqref{equ:S} is in the form of a Least Absolute Shrinkage and Selection Operator (LASSO) problem. Rewriting \eqref{equ:S}, we express the update for $\mS^{k+1}$ as
\begin{equation}\label{eq:lasso_S}
\mS^{k+1} = \argmin_\mS \left( \frac{\lambda}{\rho} \|\mS\|_1 + \frac{1}{2} \|\mH\mS - \mP\|_F^2 \right),
\end{equation}
where $\mP = -\mL^{k+1} + \mM - \mU^k$. 
To stay focused on our main content, we placed a thorough introduction of LASSO in the Appendix, where we choose 
 either FISTA (Algorithm \ref{alg:fista}) or ADMM (Algorithm \ref{alg:lasso2}) to solve it. 
Algorithm \ref{alg:gms} summarizes the ADMM steps for solving our generalized matrix separation problem.
%Together we have , using the stopping criteria as suggested in \cite{YXYJ09}.

\begin{algorithm}[htb]
\caption{Generalized Matrix Separation: GMS$(\mM_0, \mH, \lambda; \rho_O, T_O, T_I, \epsilon_O, \epsilon_I)$}\label{alg:gms}
\begin{algorithmic}
\State\textbf{Input:} $\mM_0\in\R^{m\times n}, \mH\in\R^{m\times p},  \lambda >0, \rho_O>0$, maximum number of iterations $T_O, T_I$ and the tolerance $\epsilon_O, \epsilon_I$.
\State\textbf{Output:} $\mS^{t} \in \R^{p\times n}, \mL^{t}\in\R^{m\times n}$: an approximation of solution of \eqref{equ:p} at last iteration $t$
\State\textbf{Initialize:} $\mL^0,\mU^0\in\R^{m\times n}, \mS_0 \in \R^{p \times n}$
\For{$k=0,1,\cdots,T_O-1$} 
\State 1: $ \mL^{k+1} = D_{\rho_O^{-1}}(-\mH\mS^k+\mM_0-\mU^k)$
\State 2: $\mS^{k+1} = \text{LASSO}(\mH, \mM_0 -\mU^k-\mL^{k+1}, \lambda/\rho_O; T_I, \epsilon_I) $ for Algorithm \ref{alg:fista} or 
\State \hspace{0.11in} $\mS^{k+1} = \text{LASSO}(\mH, \mM_0 -\mU^k-\mL^{k+1}, \lambda/\rho_O, \rho_I, T_I, \epsilon_I) $ for Algorithm \ref{alg:lasso2}
\State 3: $\mU^{k+1} = \mU^k + \mL^{k+1} +\mH\mS^{k+1}-\mM_0 $
\State Terminate if $\frac{\| (\mS^{k+1}, \mL^{k+1}) - (\mS^k, \mL^k) \|_F}{\|(\mS^k, \mL^k)\|_F + 1} <\epsilon_O$ 
\EndFor
\end{algorithmic}
\end{algorithm}

\begin{remark}
While alternative stopping criteria are available, such as those proposed in \cite{BS11}, 
we choose a standard stopping criterion when the change between iterations are negligible. 
%we found empirically that the fixed tolerance stopping criterion used in Algorithm~\ref{alg:gms} required fewer outer iterations in our numerical experiments. 
Moreover, implementing a dynamic update scheme for the parameter $\rho$, as recommended in \cite{BS11} and \cite{GE14}, may also be considered.
\end{remark}

\subsection{Preconditioning of $H$}\label{sec:prec}

The convergence rate of ADMM (equations \eqref{equ:admm1}--\eqref{equ:admm3}) depends on the properties of the objective functions and constraints. The work \cite{DY16} demonstrates a linear convergence rate of ADMM and its generalized version under various conditions, provided that at least one of the two objective functions is strongly convex. A linear convergence rate was also proven in \cite{admm_par} using a dynamical system approach, where strong convexity of either $f$ or $g$ is again assumed.

For a differentiable function $f: \R^n \rightarrow \R$, it is \emph{strongly convex} with a constant $\alpha>0$ if 
\begin{equation}\label{equ:str_conv}
\langle \vx-\vy, \nabla f(\vx)-\nabla f(\vy)\rangle\geq\alpha \|\vx-\vy\|_2^2, \quad\text{for all }\vx, \vy\in\R^n.
\end{equation}
This definition can be generalized to non-smooth convex functions where the gradient is replaced by the subgradient.

It is easy to verify that $f(\vx)=\frac{1}{2}\|\mH\vx-\vb\|_2^2$ is strongly convex when $\mH$ has full column rank. Given $\nabla f(\vx)=\mH^\top(\mH\vx-\vb)$, we can compute that \eqref{equ:str_conv} holds with $\alpha$ being the smallest singular value of $\mH^\top \mH$.

The results in \cite{DY16} or \cite{admm_par}  do not apply to \eqref{equ:p} since neither $\|\cdot\|_*$ nor $\|\cdot\|_1$ is strongly convex. However, they do apply to the LASSO problem because one of the objective functions is $\frac{1}{2}\|\mH\vx-\vb\|_2^2$, which is strongly convex as mentioned above. The convergence of the subproblem \eqref{equ:S} (LASSO) greatly affects the convergence of Algorithm \ref{alg:gms}. In particular, using \cite[Corollary 3.6]{DY16}, we see that the convergence rate of Algorithm~\ref{alg:lasso2} depends on the condition number of $\mH$ (the ratio of the largest to smallest positive singular values of $\mH$).
This motivates our preconditioning technique which we introduce below.

%As a subproblem  is a LASSO problem and 
%In our numerical experiments, we have utilized a preconditioning technique which  greatly improves the recovery results.

Let $\mU_H\mathbf{\Sigma}_H\mV_H^\top $ be an SVD of $\mH$, where $\mU_H\in\R^{m\times \rank(\mH)}, \mV_H\in\R^{p\times\rank(\mH)}$.  If we left multiply the constraint equation \eqref{equ:M0} with the positive semidefinite matrix $\mC=\mU_H\mathbf{\Sigma}_\mH^{-1}\mU_H^\top$, then we have 
\begin{equation}\label{equ:CH}
\mC\mM_0=\mC\mH\mS_0 + \mC\mL_0.
\end{equation} 
The new filter is $\tilde \mH =\mC\mH=\mU_\mH\mV_\mH^\top$. Like $\mH$, the preconditioned $\tilde \mH$ may be rank deficient, but $\tilde \mH$ has condition number 1. With this choice of $\mC$, $\tilde\mH$ and $\mH$ have exactly the same row space and column space, respectively.

The idea is to employ Algorithm \ref{alg:gms} with inputs $\mC\mM_0$ and $\tilde{\mH}$ (instead of $\mM_0$ and $\mH$), which will return approximations of $\mS_0$ and $\mC\mL_0$. Since $\mC$ is not necessarily invertible, recovery of $\mL_0$ (if desired) can be achieved using the constraint \eqref{equ:M0}.

To summarize, under this preconditioning strategy, the recovered sparse and low-rank components $\hat{\mS}_c$ and $\hat{\mL}_c$ are computed as
\begin{equation}\label{equ:p2}
\begin{cases}
(\hat{\mS}_c, \hat{\mY}) = \displaystyle \argmin_{\mS, \mY} \gamma \|\mS\|_1 + \|\mY\|_*, \quad \text{subject to } \mY + \tilde{\mH} \mS = \mC \mM_0, \\
\hat{\mL}_c = \mM_0 - \mH \hat{\mS}_c.
\end{cases}
\end{equation}

%We further list the ADMM steps of solving \eqref{equ:p2} in Algorithm \ref{alg:gms_c}.
The preconditioning steps are explicitly listed in Algorithm \ref{alg:gms_c} for completeness.
\begin{algorithm}[htb]
\caption{ Preconditioned Generalized Matrix Separation: PGMS$(\mM_0, \mH, \lambda; \rho_O, T_O, T_I, \epsilon_O, \epsilon_I)$}\label{alg:gms_c}
\normalsize
\begin{algorithmic}
\State\textbf{Input:} $\mM_0\in\R^{m\times n}, \mH\in\R^{m\times p},  \lambda >0, \rho_O>0$, maximum number of iterations $T_O, T_I$ and the tolerance $\epsilon_O, \epsilon_I$.
\State\textbf{Output:} $ \mS^{t} \in \R^{p\times n}, \mL^{t}\in\R^{m\times n}$: an approximation of solution of \eqref{equ:p} at last iteration $t$
\State $(\mU_H, \mathbf{\Sigma}_H, \mV_H) = \svd(\mH)$
\State $\tilde \mH = \mU_H\mV_H^\top$
\State $(\mS^t, \mY) = \text{GMS}(\mU_H\mathbf{\Sigma}_H^{-1}\mU_H^\top \mM_0, \tilde \mH, \lambda; \rho_O, T_O, T_I, \epsilon_O, \epsilon_I))$ \Comment{Algorithm \ref{alg:gms}}
\State $\mL^t = \mM_0-\mH\hat \mS^t$
\end{algorithmic}
\end{algorithm}

The work~\cite{CW25} provides incoherence conditions on $\mS_0$ and $\mL_0$ under which the recovery  by the program \eqref{equ:p} is successful. First, the filter $\mH$ must satisfy a so-called restricted infinity norm property. We restate \cite[Definition 2.1]{CW25} below.
\begin{definition}\label{def:rinp}
Given a matrix $\mS_0\in \R^{p\times n}$ and $0<\delta <1$, we say a matrix $\mH$ of dimension $m\times p$ has the $\mS_0$-$\delta$-\emph{restricted infinity norm property} ($\mS_0$-$\delta$-RINP) if 
\begin{equation}\label{equ:Hrip}
\|(\mI-\mH^\top \mH)\mA\|_\infty\leq\delta\|\mA\|_\infty \text{ for all }\mA\in\Omega(\mS_0).
\end{equation}
\end{definition}

Now we state the two quantities and the incoherence condition below.
\begin{equation}\label{equ:muH}
\mu_\mH(\mS):=\max_{\mA\in\Omega(\mS), \|\mA\|_\infty\leq1}\|\mH\mA\|.
\end{equation}

\begin{equation}\label{equ:xiH}
\xi_\mH(\mL):=\max_{\mB\in \Tc(\mL), \|\mB\|\leq1}\|\mH^\top \mB\|_\infty,
\end{equation}
where $\Tc(\mL)$ is  the tangent space at matrix $\mL$ with respect to the variety of all matrices with rank less than or
equal to rank$(\mL)$~\cite{C11}.

\begin{theorem}[{\cite[Theorem 2.7]{CW25}}]\label{thm:main}
Given $\mM_0=\mH\mS_0+\mL_0$ where  $\mH$ satisfies  \eqref{equ:Hrip} with $\delta<1/3$. If
\begin{equation}\label{equ:incoherence}
\mu_\mH(\mS_0)\xi_\mH(\mL_0)<\frac{1-3\delta}{6},
\end{equation}
then there exists $\gamma>0$ such that for any optimizer $(\hat \mS, \hat \mL)$ of \eqref{equ:p}, we must have $\hat \mS=\mS_0, \hat \mL=\mL_0$.  
\end{theorem}
To be precise, Theorem \ref{thm:main} is only a special case of \cite[Theorem 2.7]{CW25} as it requires $\mH$ itself to satisfy RINP, instead of a slightly weaker assumption. However, this is sufficient for our analysis of the preconditioned method \eqref{equ:p2}.

Let $\mP_{\ker(\mH)}$ denote the orthogonal projection onto the kernel of $\mH$.
\begin{theorem}\label{thm:prec}
Given constraint $\mM_0=\mH\mS_0+\mL_0$ and let $\tilde \mH=\mC\mH$ be the preconditioned new filter as explained in \eqref{equ:CH}. If there exists $\delta<1/3$ such that
\begin{equation}\label{equ:Hrip2}
\|\mP_{\ker(\mH)}\mA\|_\infty\leq\delta\|\mA\|_\infty \text{ for all }\mA\in\Omega(\mS_0)
\end{equation}
and
\begin{equation}\label{equ:incoherence2}
\mu_{\tilde \mH}(\mS_0)\xi_{\tilde \mH}(\mC\mL_0)<\frac{1-3\delta}{6},
\end{equation}
then there exists $\gamma>0$ such that for any optimizer $(\hat \mS_c, \hat \mL_c)$ of \eqref{equ:p2}, we must have $\hat \mS_c=\mS_0, \hat \mL_c=\mL_0$.  
\end{theorem}
\begin{proof}
As mentioned previously, $\mH=\mU_H\mathbf{\Sigma}_H\mV_H^\top$ is its SVD and $\tilde \mH =\mC\mH=\mU_H\mV_H^\top$. It is obvious that 
$\|\tilde \mH\|=1$ as all singular values of $\tilde \mH$ are 1. It is also easy to verify that $(\tilde \mH)^\dagger = (\tilde \mH)^\top$, where $(\tilde \mH)^\dagger$ is the pseudoinverse of $\tilde \mH$.

Note that $\mH$ and $\tilde \mH$ share the same null space, so the projection onto the null space of $\mH$ is
$\mP_{\ker(\mH)}= \mP_{\ker(\tilde \mH)}=\mI-\tilde \mH^\dagger \tilde \mH = \mI-\tilde \mH^\top\tilde \mH$, so \eqref{equ:Hrip2} is equivalent to $\tilde \mH$ satisfying $\mS_0$-$\delta$-RINP. We now apply Theorem \ref{thm:main} to the separation problem $\mC\mM_0=\tilde \mH\mS_0+\mC\mL_0$ given \eqref{equ:incoherence2} is satisfied. We conclude that we must have $\hat \mS_c=\mS_0, \hat \mY=\mC\mL_0$.  Finally, $\hat \mL_c=\mM_0-H\hat \mS_c=\mM_0-\mH\mS_0=\mL_0$.
\end{proof}

Note that the problems \eqref{equ:p} and \eqref{equ:p2} may not be equivalent. However, Theorem \ref{thm:prec} says under some incoherence condition, \eqref{equ:p2} will reach the desired solution, which is what we ultimately care about.

\begin{remark}
We want to emphasize that Theorem \ref{thm:prec} provides a theoretical guarantee about the convex problem \eqref{equ:p2}. It does not analyze the convergence of the preconditioned ADMM steps for solving \eqref{equ:p2}.
\end{remark}

\section{Iterative Steps for the Generalized Matrix Separation Problem: Application to Videos}\label{sec:tensor}

%To reiterate, we aim to solve the tensor optimization problem, 
%\begin{equation}\label{equ:pt}
%(\hat S, \hat L)=\argmin_{\Sc\in \R^{p_1 \times p_2 \times K}, \Lc\in\R^{m_1\times m_2 \times K}}\lambda\|\mat(\Sc)\|_1+\|\mat(\Lc)\|_*, \quad\text{subject to }L+HS=M_0.
%\end{equation}

In this section, we discuss the details of applying Algorithm \ref{alg:gms} when the given matrix $\mM_0$ originates from a 3D array, which we call  a tensor conveniently. Tensors, as higher-order arrays, are natural choices for analyzing high-dimensional structures in various applications. 

Let  $\boldsymbol{\mathcal{T}} \in \mathbb{F}^{m_1 \times m_2 \times K}$ be a three-dimensional tensor. We denote its $k$th slice along the third dimension by $\Tc[:,:,k]$. 
We define the \emph{matrix unfolding}\footnote{Note that this is different from the unfolding operator used in tensor product frameworks such as in \cite{kilmer2013third}.} of $\Tc$ along the third dimension as
\[
\mat(\boldsymbol{\mathcal{T}}) := \begin{bmatrix} \vecc(\boldsymbol{\mathcal{T}}[:,:,1]), \,  \vecc(\boldsymbol{\mathcal{T}}[:,:,2]), \,  \cdots, \,  \vecc(\boldsymbol{\mathcal{T}}[:,:,K]) \end{bmatrix} \in \mathbb{F}^{m_1 m_2 \times K}.
\]

In this notation, each column of $\mat(\boldsymbol{\mathcal{T}})$ corresponds to the vectorization of a slice of $\boldsymbol{\mathcal{T}}$. Thus, $\mat(\boldsymbol{\mathcal{T}})$ reshapes $\boldsymbol{\mathcal{T}}$ by stacking vectorized slices column-wise.  Here, we use capital letters to indicate the matrix version of tensors, such as $\mX=\mat(\Xc), \mL=\mat(\Lc)$, etc. We define $\ten(\cdot)$ to be the inverse of $\mat(\cdot)$, i.e.,
\[
\ten(\mat(\boldsymbol{\mathcal{T}})) = \boldsymbol{\mathcal{T}}.
\]

Given a known tensor $\Mc_0=\Lc_0 + \ten(\mH\mat(\Sc_0))\in\R^{m_1\times m_2 \times K}$, we aim to recover $\Sc_0\in\R^{p_1 \times p_2 \times K}$ and $\Lc_0\in\R^{m_1\times m_2 \times K}$ under the prior assumption that
 $\mL_0=\mat(\Lc_0)$ is a low rank matrix and $\mS_0=\mat(\Sc_0)$ is sparse. The filter $\mH\in\R^{m_1m_2\times p_1p_2}$ acts on the sparse tensor framewise. This setting is suitable for background removal in videos, where $\Mc_0$ is a video and each frame $\Mc_0[:,:,k]$ is a grayscale image.
The problem \eqref{equ:p} will take the form
\begin{equation} \label{eq:filtered_ten_sep}
(\hat \Sc, \hat \Lc) = \argmin_{\Sc,\Lc} 
 \lambda \| \mat(\Sc) \|_1 + \| \mat(\Lc) \|_*, \quad \text{s.t. }  \Lc + \ten(\mH\mat(\Sc)) = \Mc_0.
\end{equation}

We will detail the process of solving \eqref{eq:filtered_ten_sep} using \eqref{equ:L}--\eqref{equ:Y}. Step \eqref{equ:L} remains a singular value thresholding step, while Step \eqref{equ:S} corresponds to a tensor LASSO step, which is detailed in Section \ref{sec:tlasso}.

Even for a small video, the filter $\mH$ can be large and memory-intensive to store. In this section, we focus on the case where the linear operator $\mH\in\R^{m_1m_2\times p_1p_2}$ takes the specific form
\begin{equation}\label{equ:G2G1}
\mH = \mG_2 \otimes \mG_1,
\end{equation}
where $\mG_i \in \R^{m_i \times p_i}$, for $i = 1,2$. 
Such filters require less storage and allow faster application to data. Moreover, the setting \eqref{equ:G2G1} is common and practical: applying $\mG_2 \otimes \mG_1$ to an image $\mX$ corresponds to applying $\mG_1$ to the columns of $\mX$, followed by applying $\mG_2$ to the rows. For example, the 2D Fourier transform has the form given in \eqref{equ:G2G1}.

\begin{definition}\label{bilinear-trans}
\normalfont
For matrices $\mL \in \mathbb{F}^{p \times m_1}$ and $\mR \in \mathbb{F}^{m_2 \times q}$, we define the framewise transformation $\mathfrak{B}_{\mL, \mR}: \Fb^{m_1\times m_2\times K} \rightarrow  \Fb^{p\times q\times K}$ as
\[
\mathfrak{B}_{\mL, \mR}(\boldsymbol{\mathcal{T}})[:,:,k] = \mL\Tc[:,:,k]\mR, \quad \text{for } k = 1, 2, \dots, K,
%\mathfrak{B}_{L, R}(\mathcal{T}) := \begin{bmatrix} L \mathcal{T}[:,:,1] R, \, L \mathcal{T}[:,:,2] R, \, \cdots, \, L \mathcal{T}[:,:,k]R \end{bmatrix},
\]
 where each slice $\boldsymbol{\mathcal{T}}[:,:,k]$ is transformed by left-multiplying with $\mL$ and right-multiplying with $\mR$.
\end{definition}

The motivation for Definition \ref{bilinear-trans} is to provide an explicit expression for %$\mH\mat(\Xc)$: 
\begin{align}\label{equ:AB}
%\mH\mat(\Xc)  = 
(\mA \otimes \mB)\mat(\Xc) = \mat(\mathfrak{B}_{\mB, \mA^\top}(\Xc)).
\end{align}
The operator $\mathfrak{B}_{\mB, \mA^\top}(\Xc)$ can be easily implemented using the Matlab function \verb|pagetimes()|.
%more concise notation for the expression $\ten((G_2\otimes G_1)(\mat(\Xc)))$. We have
%\begin{equation}\label{equ:b}
%\ten(H(\mat(\Xc)))=\ten((G_2\otimes G_1)(\mat(\Xc)))=\mathfrak{B}_{G_1, G_2^\top}(\Xc),
%\end{equation}
%which is frequently used later in this section.

\begin{definition}\label{frame-div}
\normalfont
Let $\mY \in \mathbb{R}^{m_1 \times m_2}$ be a matrix and $\boldsymbol{\mathcal{X}} \in \mathbb{F}^{m_1 \times m_2 \times K}$ be a tensor. We define the slice-wise pointwise division  $\mathfrak{D}_\mY: \mathbb{F}^{m_1 \times m_2 \times K}  \to \mathbb{F}^{m_1 \times m_2 \times K}$ as
%The \emph{frame-wise  entry-wise division} of $\mathcal{X}$ by $Y$ is defined as:
\[
\mathfrak{D}_\mY(\boldsymbol{\mathcal{X}})[:,:,k] = \boldsymbol{\mathcal{X}}[:,:,k] \oslash \mY, \quad \text{for } k = 1, 2, \dots, K,
\]
where each slice $\boldsymbol{\mathcal{X}}[:,:,k]$ is pointwise divided by the matrix $\mY$.
\end{definition}

%The motivation for this structure comes from the fact that it lets transformations act separately on the rows and columns of a tensor, utilizing property \eqref{equ:vec}. This structure also has applications in machine learning, where multiple different filters are applied to a tensor. The following example illustrates this construction.
%We have 
%$$H\vecc(X) = \vecc(G_1XG_2^\top), \text{ for any }X\in\R^{p_1\times p_2}$$ so $H$ acts on each image by acting on the rows and columns.

%
%Step \eqref{equ:S} is a still a LASSO problem, which we will detail  below.

\subsection{Tensor LASSO Problem}\label{sec:tlasso}
We consider the following optimization problem:
\begin{equation}\label{equ:lasso-3d}
\hat \Xc = \argmin_{\Xc} \left(\lambda \| \mat(\Xc) \|_1 + \frac{1}{2} \| \mH \mat(\Xc) - \mat(\boldsymbol{\mathcal{B}}) \|_F^2  \right),
\end{equation} which, equivalently in matrix notation, can be written as
\begin{equation}\label{eq:lasso_St}
\hat \mX = \argmin_\mX \left( \lambda \|\mX\|_1 + \frac{1}{2} \|\mH\mX - \mB\|_F^2 \right),
\end{equation}
where $\mH=\mG_2 \otimes \mG_1$ is the coefficient matrix.

Once again, we have placed a basic introduction of LASSO in the Appendix.
If we use FISTA to solve \eqref{equ:lasso-3d}, the Lipschitz constant is   $L=\|\mH\|^2=\|\mG_2 \otimes \mG_1\|^2=\|\mG_1\|^2\|\mG_2\|^2$. The simplicity of FISTA makes the transition from regular LASSO to tensor LASSO straightforward. The only tweak is that we utilize the notation $\mathfrak{B}_{\mL, \mR}$ in the $\Xc$ update.
%The only difference is that in the $\Xc$ update, we  compute
%$$\ten(\mH^\top(\mH\mat(\Yc)-\mat(\Bc)))=\mathfrak{B}_{\mG_1, \mG_2^\top}\left(\mathfrak{B}_{\mG_1^T, \mG_2}(\Yc)-\Bc\right).$$
See Algorithm \ref{alg:lasso3dfista} for the full steps.
\begin{algorithm}[htb]
\caption{LASSO($\mG_1, \mG_2, \Bc, \lambda; T, \epsilon$): FISTA}\label{alg:lasso3dfista}
\begin{algorithmic}
\State\textbf{Input:} $\mG_1 \in \R^{p_1 \times m_1}, \mG_2 \in \R^{p_2 \times m_2}, \Bc\in\R^{p_1 \times p_2\times K}, \lambda >0, \rho>0$, maximum number of iterations $T$ and the tolerance $\epsilon$.
\State\textbf{Output:} $\Xc^{t}\in\R^{m_1 \times m_2\times K}$: an approximation of solution of \eqref{equ:lasso-3d} at last iteration $t$
\State $L=\|\mG_1\|^2\|\mG_2\|^2$.
\State\textbf{Initialize:} $\Xc^0, \Yc^0 \in \R^{m_1\times m_2 \times K}$ 

\For{$k=0,1,\cdots,T-1$} 
\State 1: $\Xc^{k+1} = S_{\lambda/L}\!\left(\Yc^k-\tfrac{1}{L}\mathfrak{B}_{\mG_1, \mG_2^\top}\left(\mathfrak{B}_{\mG_1^T, \mG_2}(\Yc^k)-\Bc\right)\right)$ 
 \State 2: $t^{k+1} = \frac{1 + \sqrt{1 + 4(t^k)^2}}{2}$
 \State 3: $\Yc^{k+1} = \Xc^{k} + \frac{t^k - 1}{t^{k+1}}(\Yc^{k+1} - \Xc^k)$
\State Terminate if $\|\Xc^{k+1}-\Xc^k\|_F/\|\Xc^k\|_F <\epsilon$ 
\EndFor
\end{algorithmic}
\end{algorithm}

If we use ADMM to solve \eqref{equ:lasso-3d}, the three steps~\eqref{equ:l1}-\eqref{equ:l3} become:
\begin{align}\label{equ:L1}
&\mX^{k+1} = (\mH^\top \mH + \rho \mI)^{-1}(\mH^\top \mB + \rho (\mZ^k - \mU^k))\\
&\mZ^{k+1} = S_{\lambda/\rho}(\mX^{k+1} + \mU^k)\\
&\mU^{k+1} = \mU^k + \mX^{k+1} - \mZ^{k+1}.
\end{align}

Solving \eqref{equ:L1} is to solve the linear system:
\begin{align}\label{lasso:3d-step1}
(\mH^\top \mH + \rho \mI)\mX = \mH^\top \mB + \rho(\mZ^k - \mU^k).
\end{align}
Let $\boldsymbol{\mathcal{R}}$ denote the tensor corresponding to the right-hand side of \eqref{lasso:3d-step1}, given by
\[
\boldsymbol{\mathcal{R}}= \ten\left(\mH^\top \mB + \rho(\mZ^k - \mU^k)\right)\stackrel{\eqref{equ:AB}}{=} \mathfrak{B}_{\mG_1^T, \mG_2}(\Bc)+ \rho(\boldsymbol{\mathcal{Z}}^k - \boldsymbol{\mathcal{U}}^k).
\]

To solve \eqref{lasso:3d-step1}, we precompute an SVD of $\mH^\top \mH+\rho$, which reduces to computing an SVD of $\mG_i^\top\mG_i$. Specifically, let $\mG_i^\top \mG_i = \mV_i \mathbf{\Sigma}_i \mV_i^\top$ for $i=1,2$. Then, the  SVD of $\mH^\top \mH$ is given by
\begin{align}
\mH^\top \mH &= \mG_2^\top \mG_2 \otimes \mG_1^\top \mG_1 = (\mV_2 \mathbf{\Sigma}_2 \mV_2^\top) \otimes (\mV_1 \mathbf{\Sigma}_1 \mV_1^\top) = (\mV_2 \otimes \mV_1)(\mathbf{\Sigma}_2 \otimes \mathbf{\Sigma}_1)(\mV_2 \otimes \mV_1)^\top. 
\end{align}
So, \eqref{lasso:3d-step1} becomes
$$(\mV_2 \otimes \mV_1)(\mathbf{\Sigma}_2 \otimes \mathbf{\Sigma}_1 + \rho)(\mV_2 \otimes \mV_1)^\top \mX=\mR, $$ and its solution is
\begin{align}\label{equ:Xt1}
\mX = (\mV_2 \otimes \mV_1)(\mathbf{\Sigma}_2 \otimes \mathbf{\Sigma}_1 + \rho)^{-1}(\mV_2 \otimes \mV_1)^\top \mR=(\mV_2 \otimes \mV_1)(\mathbf{\Sigma}_2 \otimes \mathbf{\Sigma}_1 + \rho)^{-1}(\mat(\mathfrak{B}_{\mV_1^\top, \mV_2} ( \Rc ))).
\end{align}
Since $(\mathbf{\Sigma}_2 \otimes \mathbf{\Sigma}_1 + \rho)^{-1}$ is a diagonal matrix, utilizing Definition \ref{frame-div},
\begin{align*}
&\ten[(\mathbf{\Sigma}_2 \otimes \mathbf{\Sigma}_1 + \rho)^{-1}(\mat(\mathfrak{B}_{\mV_1^\top, \mV_2} ( \Rc )))]=\mathfrak{D}_{\im(\diag(\mathbf{\Sigma}_2 \otimes \mathbf{\Sigma}_1 + \rho))}\left(\mathfrak{B}_{\mV_1^\top, \mV_2} ( \Rc )\right)\\&\stackrel{Lemma\ \ref{lem}}{=} \mathfrak{D}_{\mathbf{\Sigma} + \rho}(\mathfrak{B}_{\mV_1^\top, \mV_2} ( \boldsymbol{\mathcal{R}} )),
\end{align*}
where $\mathbf{\Sigma} = \diag(\mathbf{\Sigma}_1)[\diag(\mathbf{\Sigma}_2)]^\top $.
%so 
%\begin{align*}
%\im((\mathbf{\Sigma}_2 \otimes \mathbf{\Sigma}_1 + \rho)^{-1}C[:,i])=&\im(C[:,i])\oslash \im(\diag(\mathbf{\Sigma}_2 \otimes \mathbf{\Sigma}_1 + \rho))\\
%\stackrel{Lemma\ \ref{lem}}{=}&\im(C[:,i])\oslash (\diag(\mathbf{\Sigma}_1)[\diag(\mathbf{\Sigma}_2)]^\top+\rho).
%\end{align*}
Therefore the solution to \eqref{lasso:3d-step1} is expressed as
\begin{align}\label{svd:d}
%X = V_1\left( (V_1^\top R V_2)  \oslash (\mathbf{\Sigma}^2 + \rho)\right)V_2^\top. \\
\mathcal{X} = \mathfrak{B}_{\mV_1, \mV_2^\top} \left( \mathfrak{D}_{\mathbf{\Sigma} + \rho}(\mathfrak{B}_{\mV_1^\top, \mV_2} ( \boldsymbol{\mathcal{R}} )) \right ).
\end{align}

Algorithm \ref{alg:lasso3d} summarizes the above steps.
%The following algorithm outlines these steps for solving the tensor LASSO problem.

%\textcolor{red}{Use Lemma \ref{lem}. $\mathbf{\Sigma}=?$ below}

\begin{algorithm}[htb]
\caption{LASSO($\mG_1, \mG_2, \Bc, \lambda; \rho, T, \epsilon$): ADMM}\label{alg:lasso3d}
\begin{algorithmic}
\State\textbf{Input:} $\mG_1 \in \R^{p_1 \times m_1}, \mG_2 \in \R^{p_2 \times m_2}, \Bc\in\R^{p_1 \times p_2\times K}, \lambda >0, \rho>0$, maximum number of iterations $T$ and the tolerance $\epsilon$.
\State\textbf{Output:} $\Xc^{t}\in\R^{m_1 \times m_2\times K}$: an approximation of solution of \eqref{equ:lasso-3d} at last iteration $t$
\State\textbf{Initialize:} $\Zc^0, \Uc^0 \in \R^{m_1\times m_2 \times K}$ 
\State $(\mV_i, \mathbf{\Sigma}_i, \sim) = \svd(\mG_i^\top \mG_i), i = 1, 2$
%\State $G_1^\top G_1 = V_1 \mathbf{\Sigma}_1 V_1^\top$, $G_2^\top G_2 = V_2 \mathbf{\Sigma}_2 V_2^\top$
\State $\mathbf{\Sigma} = \diag(\mathbf{\Sigma}_1)[\diag(\mathbf{\Sigma}_2)]^\top$
\State $\Tc = \mathfrak{B}_{\mG_1^\top, \mG_2}(\Bc)$
\For{$k=0,1,\cdots,T-1$} 
\State 1: $\Rc = \Tc + \rho (\Zc^{k} -\Uc^k)$ 
\State 2: $\Xc^{k+1} = \mathfrak{B}_{\mV_1, \mV_2^\top} \left( \mathfrak{D}_{\mathbf{\Sigma} + \rho}(\mathfrak{B}_{\mV_1^\top, \mV_2} ( \Rc )) \right )$ 
\State 3: $\Zc^{k+1} = S_{\lambda/\rho}(\Xc^{k+1} + \Uc^k)$
\State 4: $\Uc^{k+1} =\Uc^k + \Xc^{k+1} - \Zc^{k+1}$
\State Terminate if $\|\Xc^{k+1}-\Xc^k\|_F/\|\Xc^k\|_F <\epsilon$ 
\EndFor
\end{algorithmic}
\end{algorithm}

\subsubsection*{$G_1$ and $G_2$ are block structured}  
%This section focuses on the case where $G_i$'s are block structured
%This section is analogous to Section \ref{sec:1circ}, but more general, where the $\mG_i$'s are block circulant (instead of circulant themselves) or block structured.  Specifically, 
This special case is used in Section \ref{sec:vid}. Let $\mE_i \in \mathbb{R}^{n_i \times n_i}$, where $n_i$ divides $m_i$, and define the matrices:
\begin{align}\label{equ:Gi}
\mG_i := \mI_{\frac{m_i}{n_i}}\otimes \mE_i\in\R^{m_i\times m_i},\quad i = 1, 2.
\end{align}

Let $k_i=m_i / n_i$ for convenience. For an image $\mX\in\R^{m_1\times m_2}$, we partition it into $k_1\times k_2$ blocks as $\mX=[\mX_{ij}]$, where each block $\mX_{ij}\in\R^{n_1\times n_2}$. Then the effect of $\mH$ on $\mX$ is highly localized as
\begin{align}\notag
\mG_1\mX\mG_2^\top&=
\begin{bmatrix}
\mE_1&&&\\
&\mE_1&&\\
&&\ddots&\\
&&&\mE_1
\end{bmatrix}
\begin{bmatrix}
\mX_{11}&\mX_{12}&\cdots&\mX_{1,k_2}\\
\mX_{21}&\mX_{22}&&\mX_{2,k_2}\\
\vdots&\vdots&\ddots&\vdots\\
\mX_{k_1,1}&\mX_{k_1,2}&&\mX_{k_1k_2}
\end{bmatrix}\begin{bmatrix}
\mE_2^\top&&&\\
&\mE_2^\top&&\\
&&\ddots&\\
&&&\mE_2^\top
\end{bmatrix}\\
&=\begin{bmatrix}
\mE_1\mX_{11}\mE_2^\top&\mE_1\mX_{12}\mE_2^\top&\cdots&\mE_1\mX_{1,k_2}\mE_2^\top\\
\mE_1\mX_{21}\mE_2^\top&\mE_1\mX_{22}\mE_2^\top&&\mE_1\mX_{2,k_2}\mE_2^\top\\
\dots&\vdots&\ddots&\vdots\\\label{equ:g1xg2t}
\mE_1\mX_{k_1,1}\mE_2^\top&\mE_1\mX_{k_1,2}\mE_2^\top&&\mE_1\mX_{k_1k_2}\mE_2^\top
\end{bmatrix}.
\end{align}

The matrices $\mE_i$ can be arbitrary. An interesting example is when each entry is positive with row sums to be 1. 

\begin{example} 
Let $\mE_1, \mE_2$ both have positive entries, with each row summing to 1. Then each pixel of $\mE_1\mX_{ij}\mE_2^\top$ is a weighted average of pixel values in $\mX_{ij}$. A very simple example is
\begin{align*}
\mE_1\mX_{ij}\mE_2^\top &=
\begin{bmatrix}0.6&0.4\\0.3&0.7\end{bmatrix}
\begin{bmatrix}a&b\\c&d\end{bmatrix}
\begin{bmatrix}0.1&0.2\\0.9&0.8\end{bmatrix} \\
&=
\begin{bmatrix}
0.1(0.6a+0.4c)+0.9(0.6b+0.4d) & 0.2(0.6a+0.4c)+0.8(0.6b+0.4d)\\
0.1(0.3a+0.7c)+0.9(0.3b+0.7d) & 0.2(0.3a+0.7c)+0.8(0.3b+0.7d)
\end{bmatrix}.
\end{align*}\qed
\end{example}

To solve this case, we apply Algorithm~\ref{alg:lasso3d}. The matrices 
$\boldsymbol{\Sigma}, \mV_1$, and $\mV_2$ are computed as follows:
\begin{align*}
(\mW_i, \mS{e_i}, \sim) &= \svd(\mE_i^\top \mE_i),\quad i = 1, 2,\\
\mathbf{\Sigma} &= \diag(\mI_{k_1}\otimes \mS{e_1})[\diag(\mI_{k_2}\otimes \mS{e_2})]^\top,\\
\mV_i &=\mI_{k_i}\otimes \mW_i, \quad i = 1, 2.
\end{align*}

\subsection{Solution to the Tensor Separation Problem }
 
We now return to the solution of steps \eqref{equ:L}-\eqref{equ:Y}. The solution to \eqref{equ:L} remains a singular value thresholding step, in which the operator $\mat(\cdot)$ is applied to reshape the tensor into a matrix. The solution to \eqref{equ:S} corresponds to a LASSO step, which is detailed in Section \ref{sec:tlasso}. See Algorithm \ref{alg:gts}.

\begin{algorithm}[H]
\caption{Generalized Tensor Separation: GTS$(\Mc_0, \mG_1, \mG_2, \lambda; \rho_O, T_O, T_I, \epsilon_O, \epsilon_I)$}\label{alg:gts}
\begin{algorithmic}
\State\textbf{Input:} $\Mc_0 \in\R^{m_1\times m_2 \times K}, \mG_1 \in \R^{p_1 \times m_1}, \mG_2 \in \R^{p_2 \times m_2},  \lambda >0, \rho_O>0$, maximum number of iterations $T_O, T_I$ and the tolerances $\epsilon_O, \epsilon_I$.
\State\textbf{Output:} $ \Sc^{t} \in \R^{p_1 \times p_2 \times K}$, $\Lc^{t}\in\R^{m_1\times m_2 \times K}$: an approximation of solution of \eqref{eq:filtered_ten_sep} at last iteration $t$
\State\textbf{Initialize:} $\Lc^0, \Uc^0\in\R^{m_1\times m_2 \times K}, \Sc^{0} \in \R^{p_1 \times p_2 \times K}$
\For{$k=0,1,\cdots,T_O-1$} 
\State 1: $ \Lc^{k+1} = \ten \left(D_{\rho^{-1}}(-\mat(\mathfrak{B}_{\mG_1, \mG_2^\top}(\Sc^k))         +\mat(\Mc_0)-\mat(\Uc^k))\right) $
\State 2: $\Sc^{k+1} = \text{LASSO}(\mG_1, \mG_2, \Mc_0-\Uc^k-\Lc^{k+1}, \lambda/\rho_O; T_I, \epsilon_I)$ for Algorithm \ref{alg:lasso3dfista} or
\State \hspace{0.11in} $\Sc^{k+1} = \text{LASSO}(\mG_1, \mG_2, \Mc_0-\Uc^k-\Lc^{k+1}, \lambda/\rho_O; \rho_I, T_I, \epsilon_I) $ for  Algorithm \ref{alg:lasso3d}
\State 3: $\Uc^{k+1} = \Uc^k + \Lc^{k+1} + \mathfrak{B}_{\mG_1, \mG_2^\top}(\Sc^{k+1})-\Mc_0 $
\State Terminate if $\frac{\| (\Sc^{k+1}, \Lc^{k+1}) - (\Sc^k, \Lc^k) \|_F}{\|(\Sc^k, \Lc^k)\|_F + 1} <\epsilon_O$ 
\EndFor
\end{algorithmic}
\end{algorithm}

%\begin{remark}
%Before providing input to the algorithm, we check if $G_1$ and $G_2$ are block circulant. If not, we use Algorithm \ref{alg:lasso3d} in Step 2. Otherwise, since $G_1$ and $G_2$ consist of blocks of $E_1$ and $E_2$, we input these into Algorithm \ref{alg:lasso3d-cir} in Step 2. \textcolor{red}{reword}
%\end{remark}

\subsubsection{Preconditioning}
The preconditioning of $\mH = \mG_2\otimes \mG_1$ reduces to the separate preconditioning of $\mG_1$ and $\mG_2$.

Let $\mG_i = \mU_i \mathbf{\Sigma}_i \mV_i^\top$ be its SVD, define $\mC_i=\mU_i \mathbf{\Sigma}_i^{-1} \mU_i^\top$, and $\mC=\mC_2 \otimes \mC_1$. Then the new filter is
$$\tilde \mH=\mC\mH=(\mC_2\mG_2) \otimes (\mC_1\mG_1).$$

Similar to Section \ref{sec:prec}, we apply  $\mC_2 \otimes \mC_1$ on the left to the original constraint
$$\mat(\Mc_0)=(\mG_2\otimes \mG_1)\mat(\Sc_0)+\mat(\Lc_0)$$
which becomes
\begin{equation}\label{equ:prec2d}
(\mC_2 \otimes \mC_1) \mat(\Mc_0) = \tilde \mH \mat(\Sc_0) + (\mC_2 \otimes \mC_1) \mat(\Lc_0).
\end{equation}
We first solve for $\Sc_0$ from \eqref{equ:prec2d}, which is performed by Algorithm \ref{alg:gts_c}.

\begin{algorithm}[H]
\caption{ Preconditioned Generalized Tensor Separation: PGTS$(\Mc_0, \mG_1, \mG_2, \lambda; \rho_O, T_O, T_I, \epsilon_O, \epsilon_I)$}\label{alg:gts_c}
\begin{algorithmic}
\State\textbf{Input:} $\Mc_0 \in\R^{m_1\times m_2 \times K}, \mG_1 \in \R^{p_1 \times m_1}, \mG_2 \in \R^{p_2 \times m_2},  \lambda >0, \rho_O>0$, maximum number of iterations $T_O, T_I$ and the tolerance $\epsilon_O, \epsilon_I$.
\State\textbf{Output:} $ \Sc^{t} \in \R^{p_1 \times p_2 \times K}$, $\Lc^{t}\in\R^{m_1\times m_2 \times K}$: an approximation of solution of \eqref{eq:filtered_ten_sep} at last iteration $t$
\State $(\mU_i, \mathbf{\Sigma}_i, \mV_i) = \svd(\mG_i)$
\State $\tilde \mG_i = \mU_i\mV_i^\top$
\State $\mC_i = \mU_i\mathbf{\Sigma}_i^{-1}\mU_i^\top$
\State $(\Sc^t, \boldsymbol{\mathcal{Y}}) = \text{GTS}(\mathfrak{B}_{\mC_1, \mC_2^\top}(\Mc_0), \tilde \mG_1, \tilde \mG_2, \lambda; \rho_O, T_O, T_I, \epsilon_O, \epsilon_I)$ \Comment{Algorithm \ref{alg:gts}}
\State $\Lc^t = \Mc_0 - \mathfrak{B}_{\mG_1, \mG_2^\top}(\Sc^t)$
\end{algorithmic}
\end{algorithm}

\section{Numerical Experiments} \label{sec:numex}

We conduct various numerical experiments in this section using both synthetic and real data. For synthetic data, the ground truth matrices $\mL_0$ and $\mS_0$ are randomly generated according to the models in Table \ref{tab:2d}, where the sparse matrix $\mS_0$ is generated using three different choices of random models, where $\rho_r$ and $\rho_s$ indicates low rank ratio and sparsity ratio respectively.

\begin{table}[htbp] % not [H]
\caption{Generation of ground truth matrices $\mL_0$ and $\mS_0$}
\label{tab:2d}
\renewcommand{\arraystretch}{1.3}
\begin{tabularx}{\textwidth}{p{2.5cm} X}
\hline
\textbf{Matrix} & \textbf{Generation Method} \\
\hline
$\mL_0 \in \mathbb{R}^{m \times n}$ 
& Choose rank ratio $\rho_r$. Let $r = \lfloor \rho_r \cdot \min(m,n) \rfloor$.
Generate $\mU \in \mathbb{R}^{m \times r}$ and $\mV \in \mathbb{R}^{n \times r}$ with i.i.d.\ $\mathcal{N}(0,1)$ entries. 
Set $\mL_0 = \mU \mV^\top$. \\[4pt]

$\mS_0 \in \mathbb{R}^{p \times n}$ 
& Choose sparsity ratio $\rho_s$.
Randomly select $\lfloor \rho_s \cdot p n \rfloor$ entries to be nonzero, uniformly over the matrix.
Assign values using one of the following distributions:
Gaussian $\mathcal{N}(0,1)$; 
Uniform $\text{Unif}(-a, a)$ with $a = \|\mL_0\|_\infty$;
Impulsive $a \cdot \text{sign}(\mathcal{N}(0,1))$ with $a = \|\mL_0\|_\infty$. \\[2pt]
\hline
\end{tabularx}
\end{table}

To assess recovery accuracy, we use the relative Frobenius norm error defined as

\begin{equation}
\text{RelErr}(\mL_0, \hat \mL) = \frac{\|\hat \mL - \mL_0\|_F}{\|\mL_0\|_F}, \quad
\text{RelErr}(\mS_0, \hat \mS) = \frac{\|\hat \mS - \mS_0\|_F}{\|\mS_0\|_F},
\end{equation}
where $\hat \mL$ and $\hat \mS$ denote the recovered low-rank and sparse components, respectively.

We consistently use $\lambda=\frac{1}{\sqrt{\max\{m, n\}}}$ as the penalization parameter, following the recommendation in \cite{CLMW11}. Other parameters used in our algorithms are explained in Section \ref{sec:iter1d} and \ref{sec:tensor}. For completeness, we list these parameters again below, along with recommended values:

\begin{itemize}
\item $\rho_O$: step-size parameter for ADMM (outer loop). Choosing $\rho_O \in \{0.5, 1\}$ is recommended by \cite{BS11, GE14} and validated by our experiments.
\item $\rho_I$: step-size parameter for LASSO (inner loop) if solved by ADMM. We also recommend choosing $\rho_I \in \{0.5, 1\}$.
\item $T_O$: maximum number of iterations for ADMM.
\item $T_I$: maximum number of iterations for LASSO. We recommend choosing a relatively small value, typically on the order of tens.
\item $\epsilon_O$: convergence tolerance for ADMM. The value depends on users' accuracy needs.
\item $\epsilon_I$: convergence tolerance for LASSO. The value depends on users' accuracy needs.
\end{itemize}
Experiments are run on a MacBook Pro with Apple M3 chip and 8GB RAM, using Matlab 2023b. The code can be found in the github repository \url{https://github.com/xuemeic/matrix_sep}.

\subsection{Preconditioning vs.\ No Preconditioning}\label{sec:pvn}

For the first experiment, we generated a low-rank matrix $\mL_0\in\R^{100\times 100}$ modeled according to Table \ref{tab:2d} with rank $r=5$ ($\rho_r=5\%$), and a sparse component $\mS_0\in\R^{100\times 100}$ with $\rho_s=10\%$ nonzero entries, where each entry is drawn from the standard normal distribution.
The filter matrix $\mH$ was generated randomly, with entries i.i.d.\ from the standard normal distribution. 
We compared the performance of Algorithm \ref{alg:gms} and the preconditioned Algorithm \ref{alg:gms_c} using the same parameters: $\rho_O=\rho_I=1, T_O=60$, and $ T_I = 60$. The LASSO step is solved by ADMM.
 As shown in Figure~\ref{fig:prec_convergence}, the preconditioned version converges significantly faster.
More importantly, the preconditioned Algorithm \ref{alg:gms_c} is able to achieve accuracy on the order of $10^{-10}$, whereas Algorithm \ref{alg:gms} stalls early in the iteration. This demonstrates the effectiveness of preconditioning.
 
\begin{figure}[htb]
\centering
\includegraphics[width=0.9\textwidth]{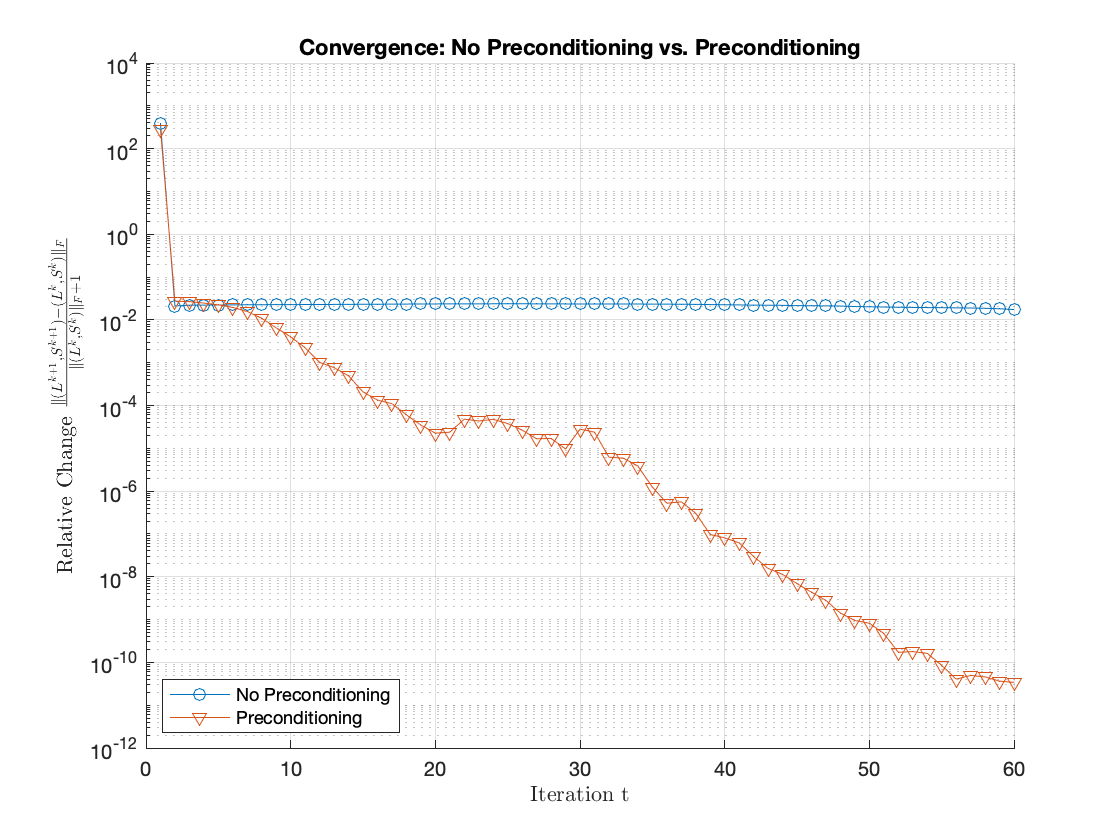}
\caption{Convergence of Algorithm~\ref{alg:gms} (no preconditioning) vs. Algorithm \ref{alg:gms_c} (with preconditioning).}
\label{fig:prec_convergence}
\end{figure}

In a second experiment, where $\mL_0, \mS_0$, and $\mH$ are modeled the same as above but with $n = 300, m = 270, p = 266$, and  $\rho_r= \rho_s=5\%$, we select $\rho_O=\rho_I=1, \epsilon_O = 10^{-7}, \epsilon_I=10^{-5}, T_O=500$, and  $T_I = 30$. The comparative performance results are shown in Table \ref{tab:con2}.
\begin{table}[htb]
\caption{$\mH\in\R^{270\times266}$ is i.i.d $N(0, 1)$}\label{tab:con2}
\begin{tabular}{rllllll}
&RelErr of $S_0$&RelErr of $L_0$ & Iterations ran & Time\\\hline
Algorithm \ref{alg:gms}: No preconditioning \\
LASSO by FISTA&\verb"3.93e+14"&\verb"7.89e+16"&500 (max) &44.45 sec\\
LASSO by ADMM & \verb"3.13e-03" & \verb"8.63e-01" & 500 (max) &34.81 sec\\

Algorithm \ref{alg:gms_c}: \textbf{With preconditioning} \\
\textbf{LASSO by FISTA} &\verb"4.58e-07" & \verb"1.92e-04"& 100&2.52 sec\\
LASSO by ADMM &\verb"8.27e-07" & \verb"3.37e-04"& 100&10.63 sec\\
\hline
\end{tabular}
\end{table}

We observe in Table \ref{tab:con2} that preconditioning is crucial to obtain successful recovery, whether LASSO was solved by FISTA or ADMM. Moreover, the best result, in terms of both accuracy and efficiency, is produced when LASSO was solved by FISTA with preconditioning. Without preconditioning, it is surprising that solving LASSO by FISTA did not converge to the true solution while solving LASSO by ADMM generated satisfying results, suggesting the higher sensitivity to conditioning when using FISTA. This can be partially seen in the convergence rate established in \cite{BATM2009}.

In a third experiment, $\mL_0$ and $\mS_0$ follow the same random models as above with parameters $n = 300, 
m = n - 1, p = m$, and  $\rho_r= \rho_s=5\%$. We use a deterministic circulant matrix $\mH$ defined in \eqref{equ:H}, and set $\rho_O=1, \rho_I=1, \epsilon_O = 10^{-7}, \epsilon_I=10^{-5}, T_O=500$, and $T_I = 30$:
\begin{equation}\label{equ:H}
\mH =\begin{bmatrix}
-1&1&&&\\
&-1&1&&\\
&&-1&\ddots&\\
&&&\ddots&1\\
1&&&&-1
\end{bmatrix}.
\end{equation}
The results are summarized in Table \ref{tab:6.1.3}. Similar to the second experiment, preconditioning benefited both  methods, and the best performance is obtained by preconditioning with LASSO solved by FISTA. However, with this circulant $\mH$, FISTA outperforms ADMM when there is no preconditioning.
\begin{table}[htb]
\caption{$H\in\R^{300\times 300}$ is circulant as in \eqref{equ:H}}\label{tab:6.1.3}
\begin{tabular}{rllllll}
&RelErr of $S_0$&RelErr of $L_0$ & Iterations ran & Time\\\hline
Algorithm \ref{alg:gms}: No preconditioning \\
LASSO by FISTA&\verb"8.94e-07"&\verb"6.38e-06"&85&8.85 sec\\
LASSO by ADMM & \verb"2.47e-03" & \verb"1.15e-02" & 500 (max) &23.82 sec\\

Algorithm \ref{alg:gms_c}: \textbf{With preconditioning} \\
\textbf{LASSO by FISTA} &\verb"2.18e-08" & \verb"8.30e-07"& 28&1.28 sec\\
LASSO by ADMM &\verb"2.14e-06" & \verb"2.38e-05"& 28&1.74 sec\\
\hline
\end{tabular}
\end{table}

\begin{remark}
When LASSO is solved by ADMM, we can utilize the circulant structure of $\mH$ when solving \eqref{equ:l1} as the coefficient matrix can be diagonalized by the Fourier transform. This will speed up Algorithm \ref{alg:lasso2}. However, with preconditioning, the new filter $\tilde \mH$ is in general no longer circulant even if $\mH$ was. Our numerical experiment indicates \textbf{the benefit of preconditioning outweighs the benefit of circulant structure}, so Algorithm \ref{alg:gms_c} should still be utilized even when $\mH$ is circulant.
\end{remark}

The fourth experiment explores the effect of $\rho_O$ and $\rho_I$ on the convergence of Algorithm \ref{alg:gms} and Algorithm \ref{alg:gms_c} while LASSO is solved by ADMM. We set $n = 100, m = n - 1$, and $p = m$. The matrices $\mL_0$ and $\mS_0$ are generated according to Table \ref{tab:2d}, with $\mS_0$ following the Gaussian model. Let $\mH$ be as defined in \eqref{equ:H}, and set $\rho_r= \rho_s=5\%, \epsilon_O = 10^{-7}, \epsilon_I=10^{-5}, T_O=100$, and $T_I = 20$. We vary $\rho_O$ over the list [0.01, 0.05, 0.1, 1, 5, 10, 50, 100, 500] and $\rho_I$ over  [0.05, 0.1, 1, 5, 10, 50, 100] for both algorithms. Figure \ref{fig:rho} shows the recovery results. For Algorithm \ref{alg:gms}, only a limited range of these parameters allows successful recovery, whereas the preconditioned Algorithm \ref{alg:gms_c} achieves successful recovery across a much wider range of parameter choices.

\begin{figure}[htb]
\includegraphics[width = \textwidth]{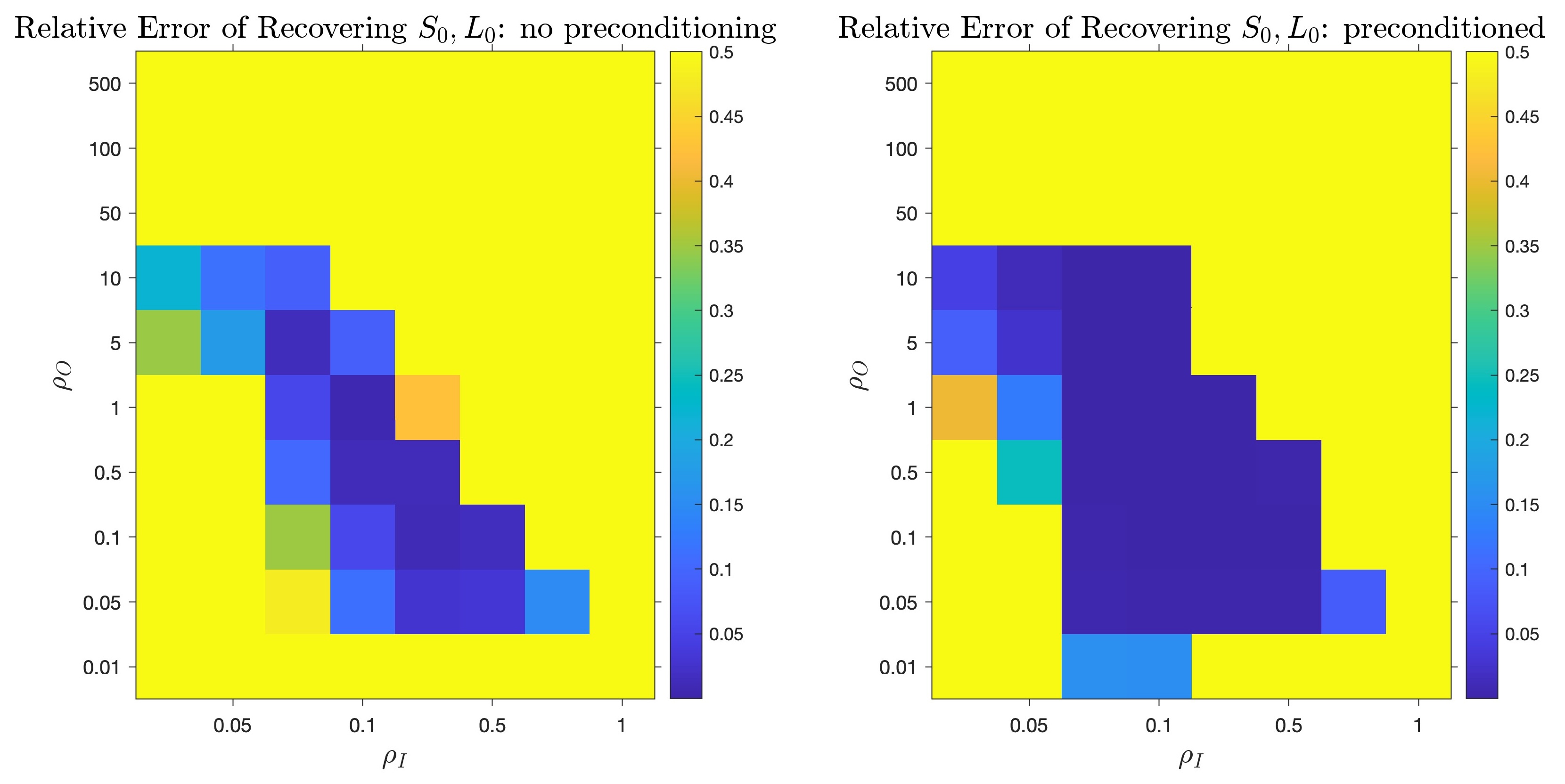}
\caption{Preconditioning (right figure) allows more robust and wider range of choices of $\rho_O, \rho_I$.}
\label{fig:rho}
\end{figure}

\textbf{Summary:} These four experiments demonstrate the superior performance of the preconditioning technique in terms of efficiency, accuracy, and robustness. We will employ this preconditioning approach in all subsequent experiments. Moreover, when preconditioned, using FISTA for LASSO consistently outperforms ADMM for LASSO. Therefore we recommend Algorithm~\ref{alg:gms_c} (combined with Algorithm \ref{alg:fista} FISTA)  for the general matrix separation problem and Algorithm~\ref{alg:gts_c} (combined with Algorithm \ref{alg:lasso3dfista} FISTA) for the specialized tensor case.

\subsection{Explorations on rank and sparsity tolerance}\label{sec:rs}
In this section, we explore the maximum rank ratio and sparsity ratio that allow successful recovery of $\mL_0$ and $\mS_0$ using Algorithm \ref{alg:gms_c} while LASSO is solved by ADMM. We set $m=n=p=100$. The following parameter values are held constant throughout all experiments and trials, as preliminary tests demonstrated their effectiveness and stability:
\[
\rho_{O} = 0.5, \quad
\rho_{I}= 1, \quad
\epsilon_O=\epsilon_I = 10^{-7}, T_O=T_I=50.
\]

We evaluated the method's performance across a range of target ranks $r$ and sparsity  levels $s$. For each $(r, s)$ pair, we ran 10 independent trials. In each trial, we generated a ground truth low-rank matrix $\mL_0$ and a sparse matrix $\mS_0$ as described in Table~\ref{tab:2d}. 
A trial was considered successful if both $\text{RelErr}(\mL_0, \hat \mL)$ and $\text{RelErr}(\mS_0, \hat \mS)$ were below a fixed threshold $\varepsilon \in \{10^{-3}, 10^{-4}\}$, depending on the experiment.
To ensure comparability across all $(r, s)$ pairs and trials, the matrix $\mH$ is fixed for the duration of the experiments. 

\subsubsection{Random Filter}
%present an experiment following a setup similar to that of \cite{YXYJ09}. 
In the first setting, the entries of $\mH$ are drawn independently from the standard normal distribution, i.e., $\mH_{ij} \sim \mathcal{N}(0, 1)$ i.i.d.

We tested different sparse models to understand how the method performs under various types of corruption.  The Gaussian model represents dense noise; the uniform model corresponds to noise that is bounded and symmetric around zero. The impulsive model simulates large outliers by assigning fixed-magnitude values with random signs, scaled relative to the largest entry in $\mL_0$. See Table \ref{tab:2d}. The setup is similar to that of~\cite{YXYJ09}. 

The results are in Figure \ref{fig:Hgauss}, where the Gaussian sparsity model  allows the largest rank and sparsity, and the impulsive model shows the worst performance. Interestingly, this result is opposite to that of \cite{YXYJ09}, where $\mH$ is the identity.

\begin{figure}[htbp]
\begin{tabular}{cc}
 \multicolumn{2}{c}{(a) $\mS_0$ follows Gaussian model} \\
 \includegraphics[width=0.48\textwidth]{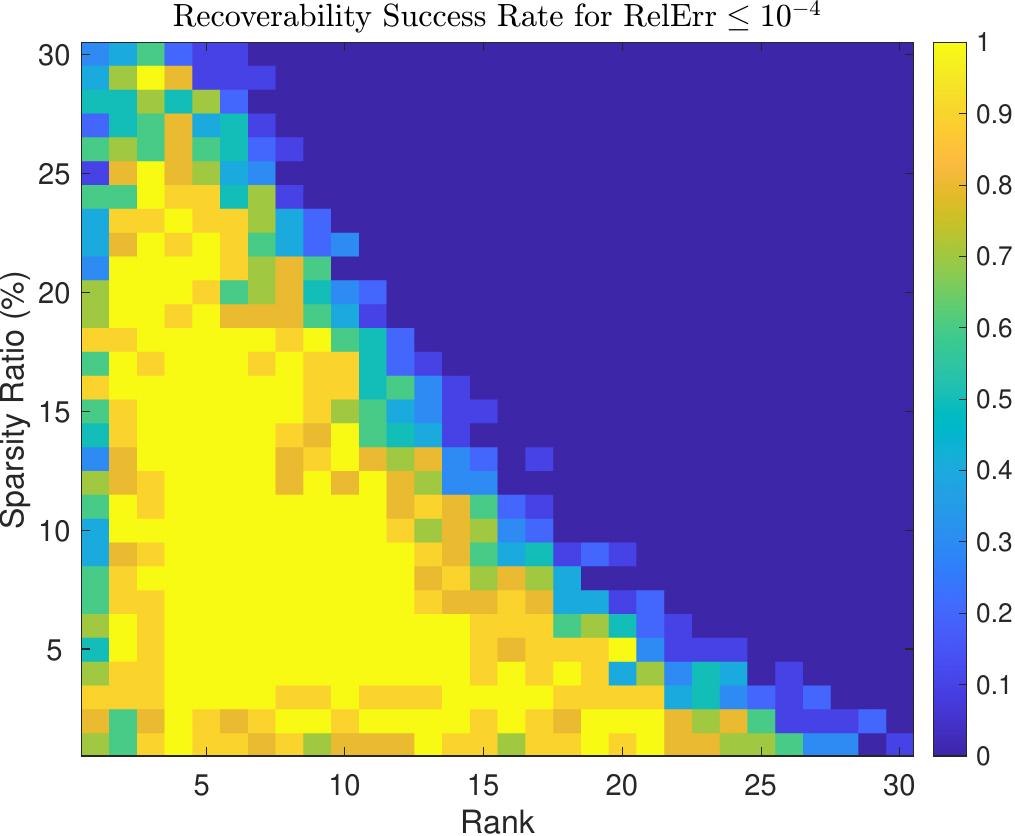} & \includegraphics[width=0.48\textwidth]{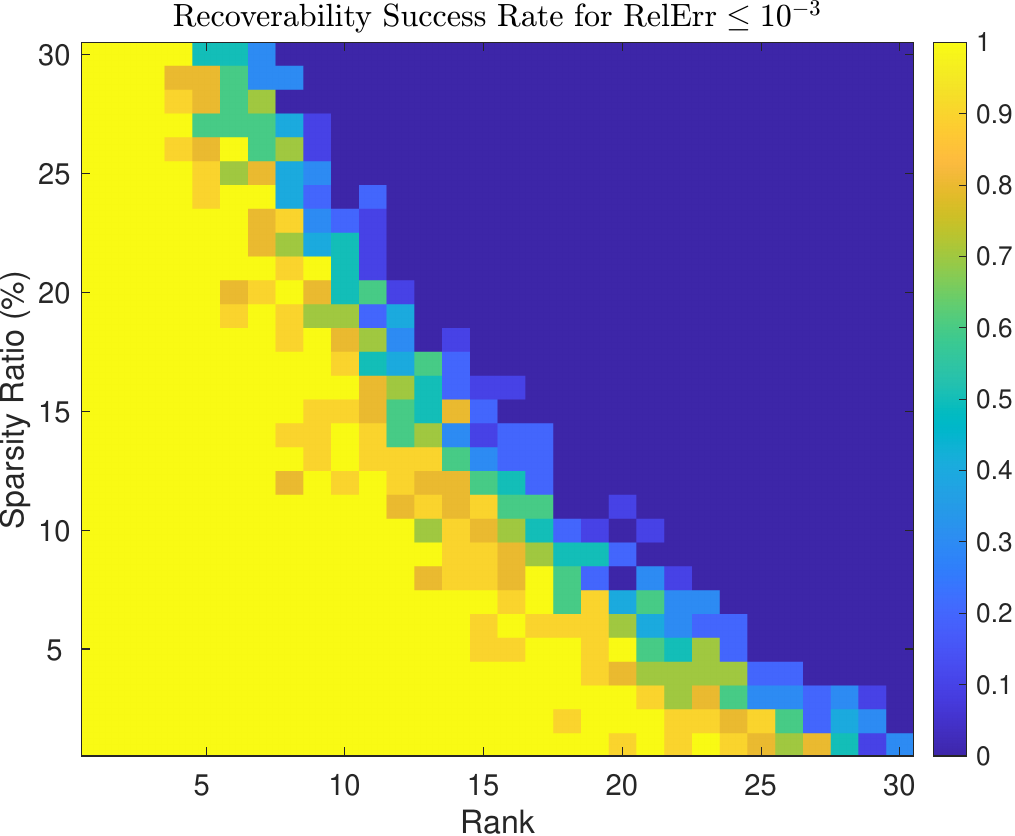}\\
\\
 \multicolumn{2}{c}{(b) $\mS_0$ follows uniform model} \\
  \includegraphics[width=0.48\textwidth]{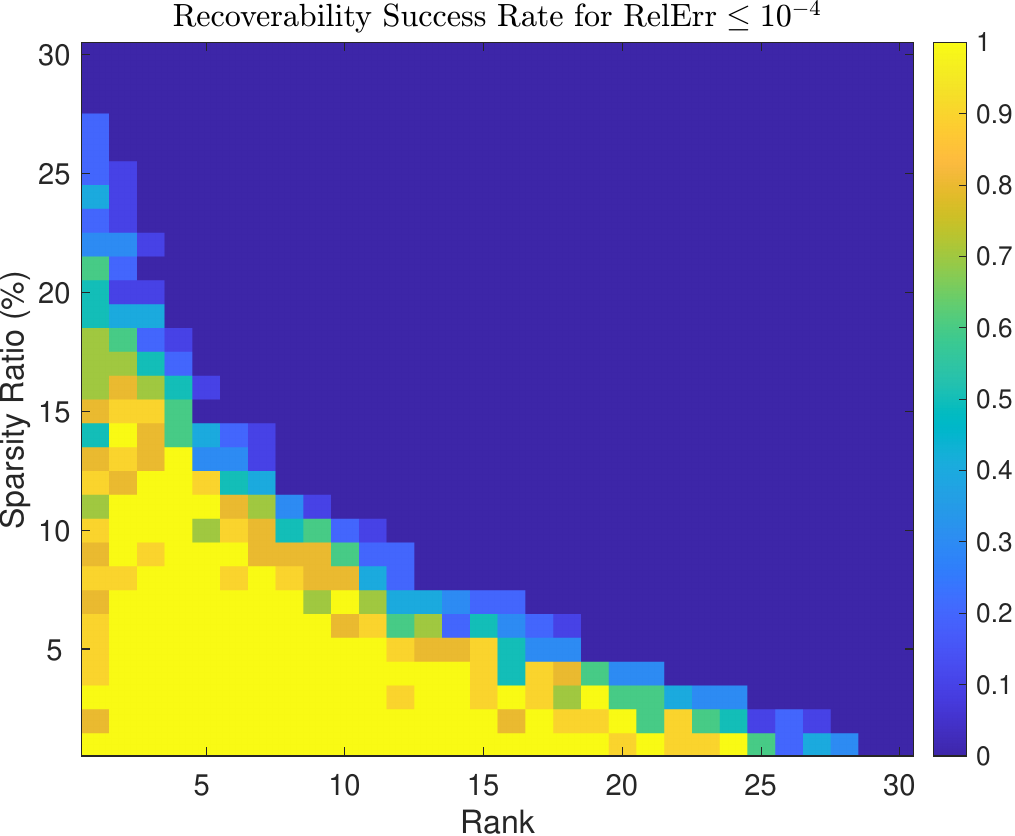} & \includegraphics[width=0.48\textwidth]{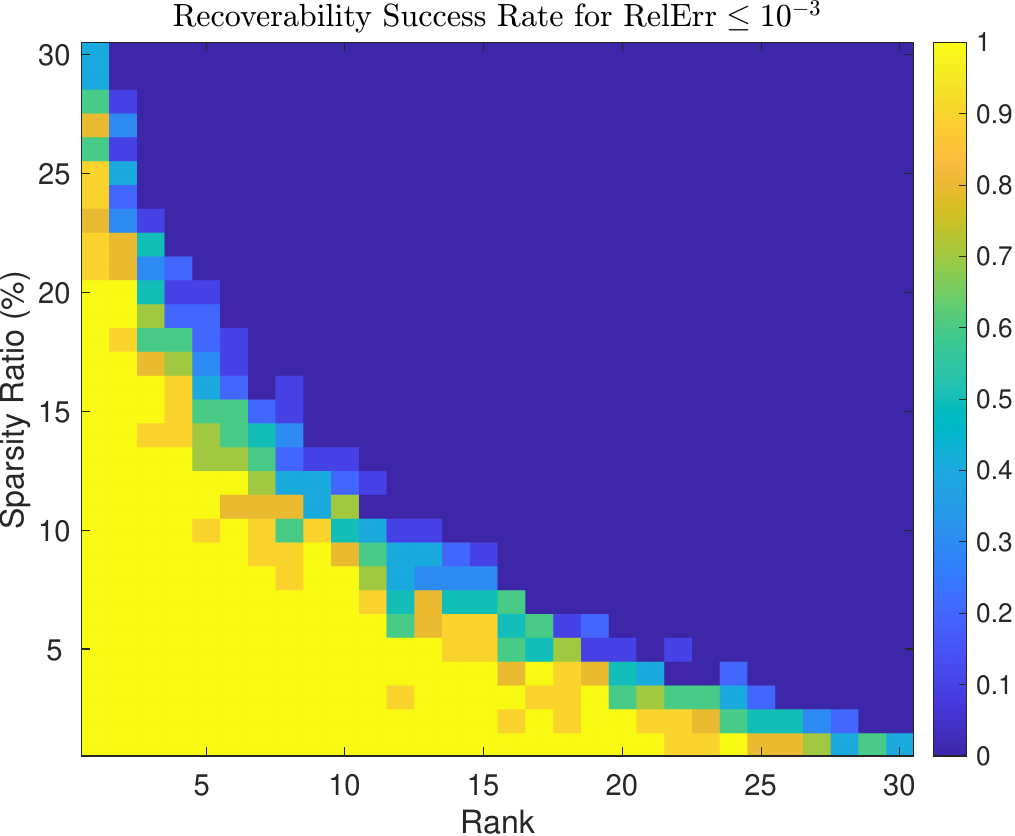}\\
  \\
   \multicolumn{2}{c}{(c) $\mS_0$ follows impulsive model} \\
   \includegraphics[width=0.48\textwidth]{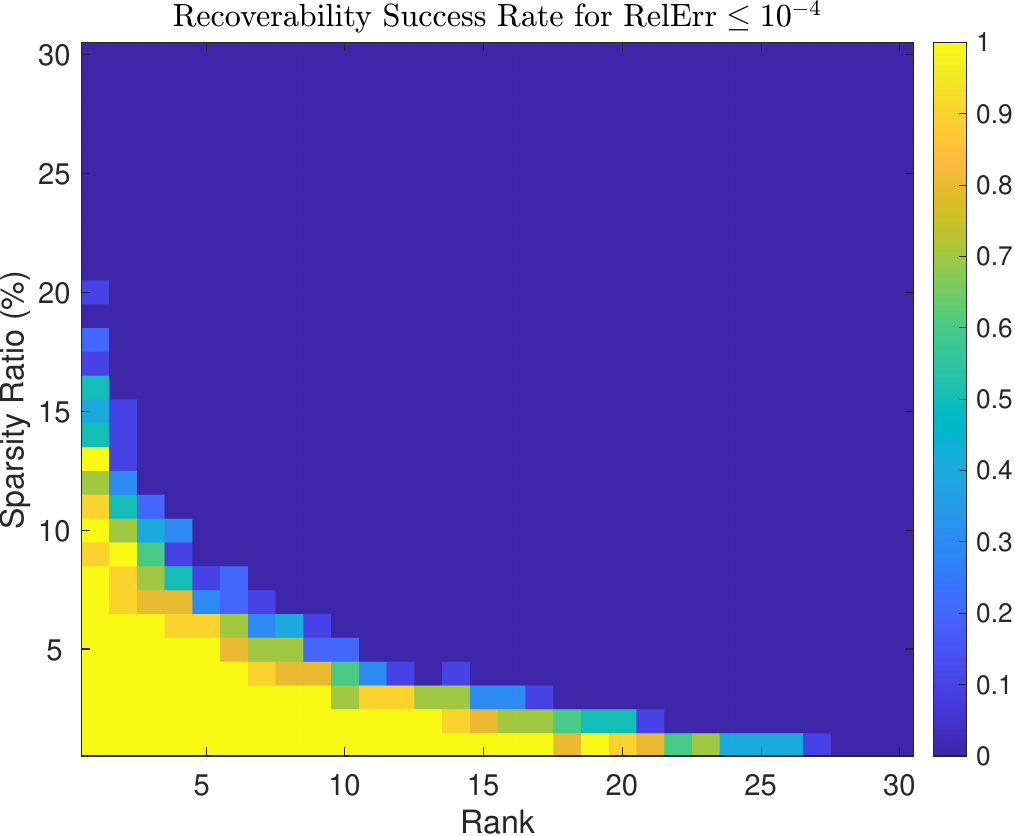} & \includegraphics[width=0.48\textwidth]{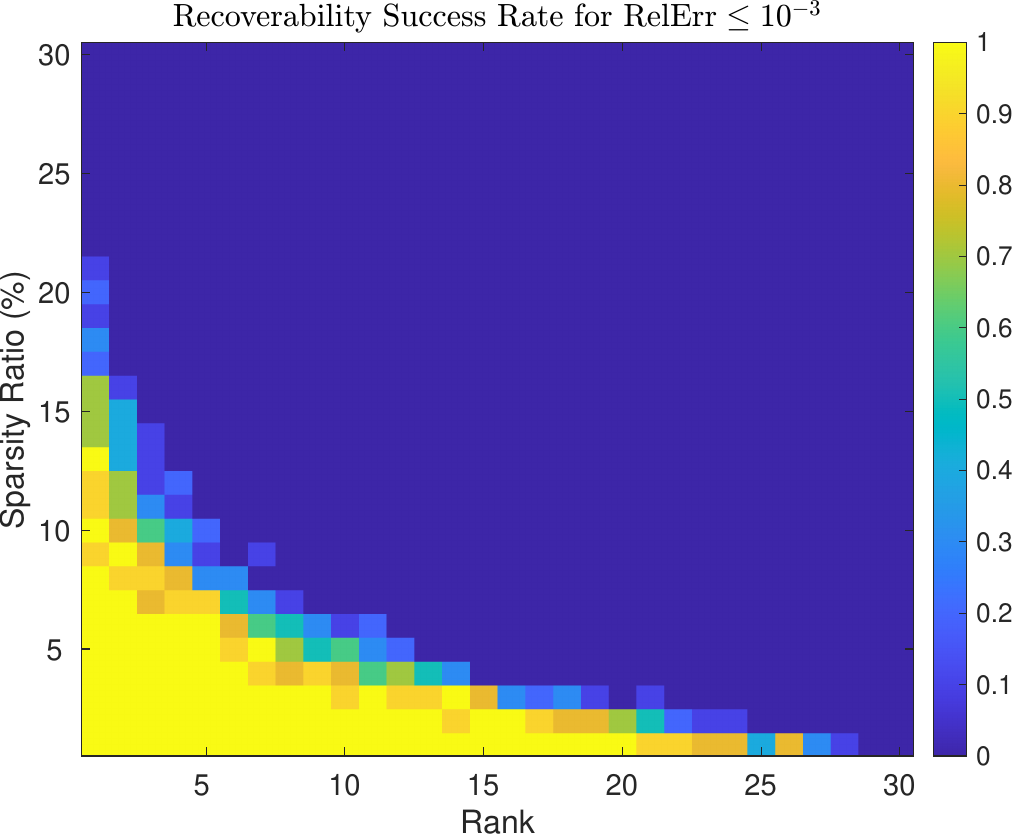}
  \end{tabular}
\caption{Recoverability results when $\mH$ is Gaussian}\label{fig:Hgauss}
\end{figure}

\subsubsection{Circulant Filter}

In the second setting, $\mH$ is taken to be the circulant matrix in \eqref{equ:H}. The results are shown in Figure \ref{fig:Hcirc}, where a similar pattern to that in Figure \ref{fig:Hgauss} is observed.

\begin{figure}[htbp]
\begin{tabular}{cc}
 \multicolumn{2}{c}{(a) $\mS_0$ follows Gaussian model} \\
 \includegraphics[width=0.48\textwidth]{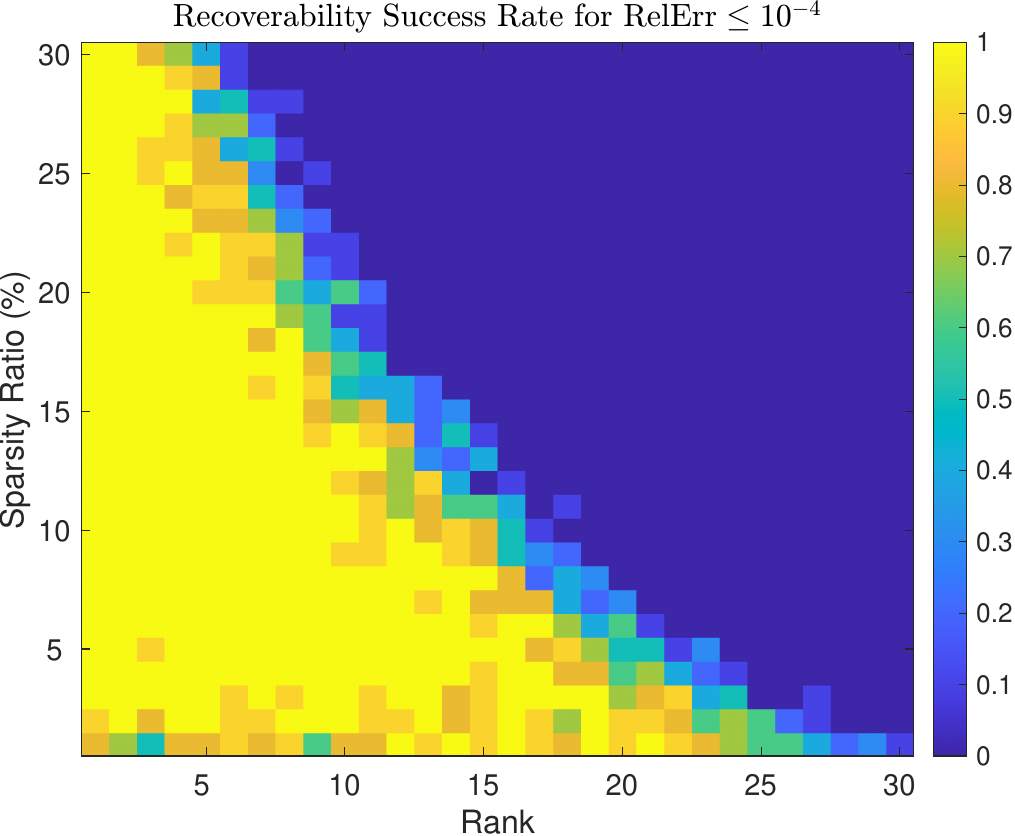} & \includegraphics[width=0.48\textwidth]{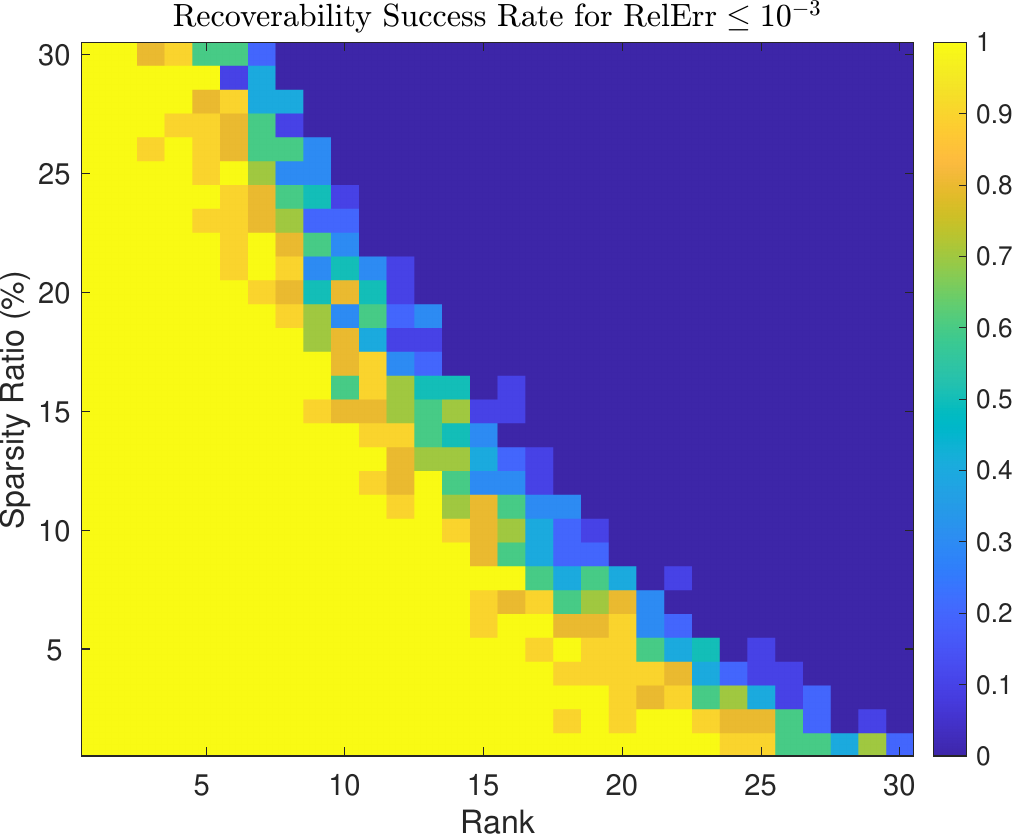}\\
\\
 \multicolumn{2}{c}{(b) $\mS_0$ follows uniform model} \\
  \includegraphics[width=0.48\textwidth]{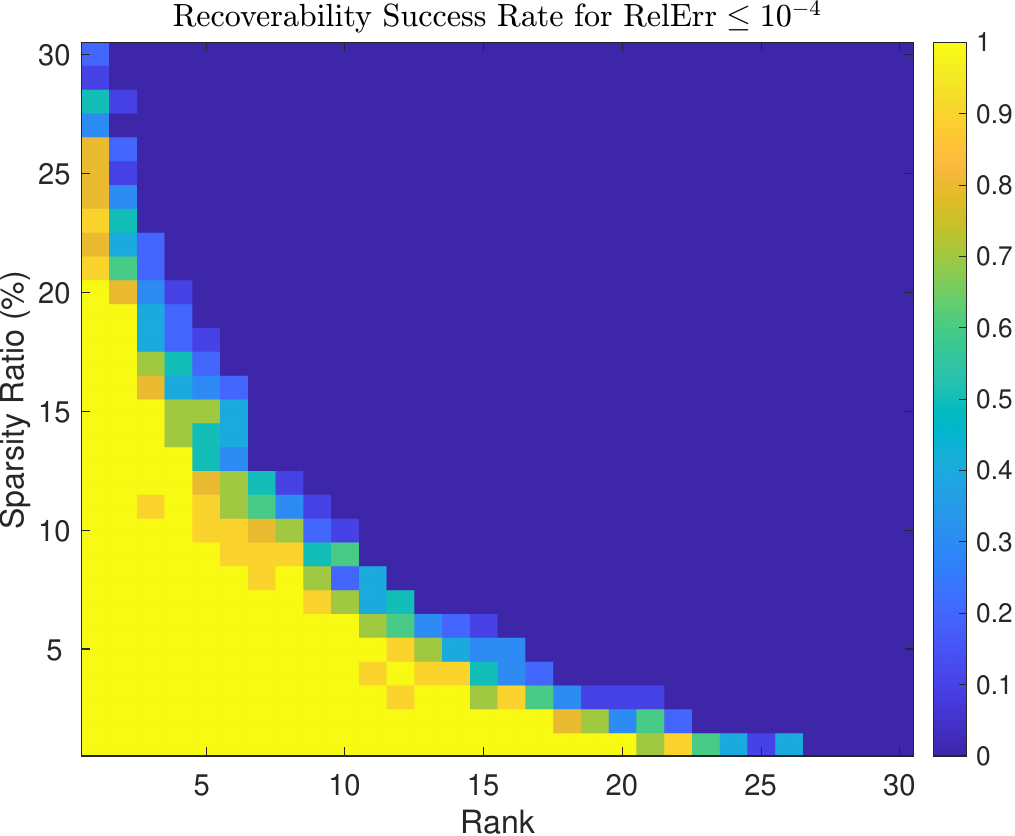} & \includegraphics[width=0.48\textwidth]{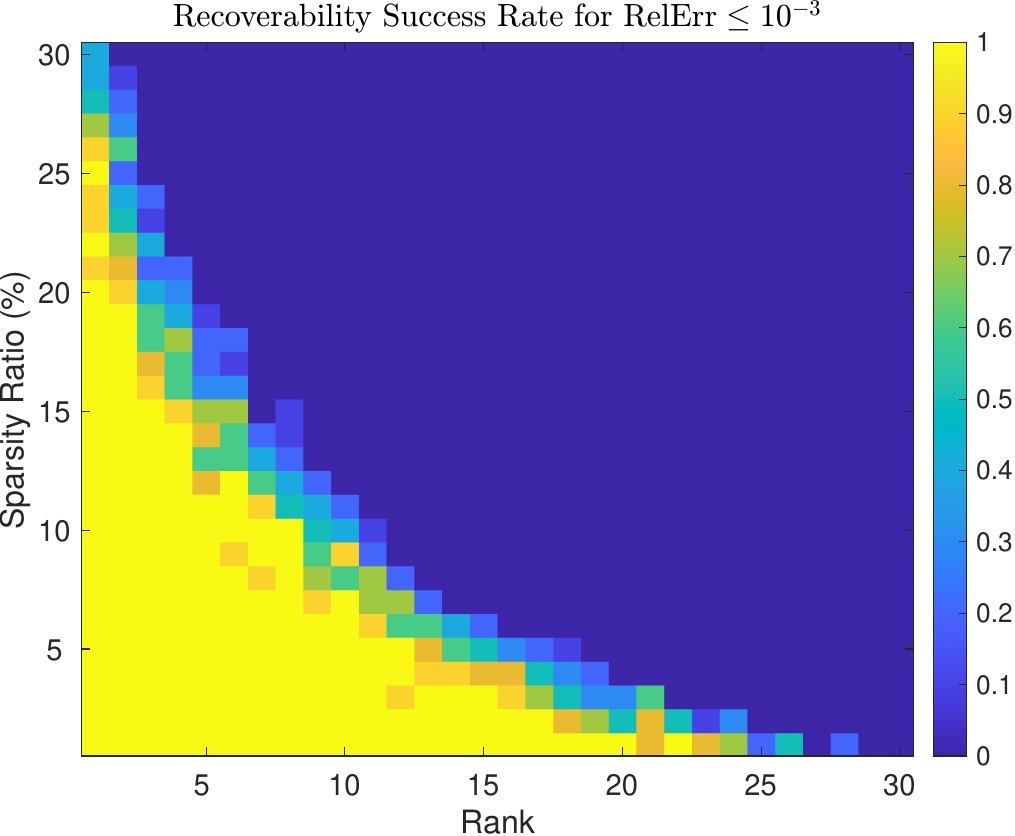}\\
  \\
   \multicolumn{2}{c}{(c) $\mS_0$ follows impulsive model} \\
   \includegraphics[width=0.48\textwidth]{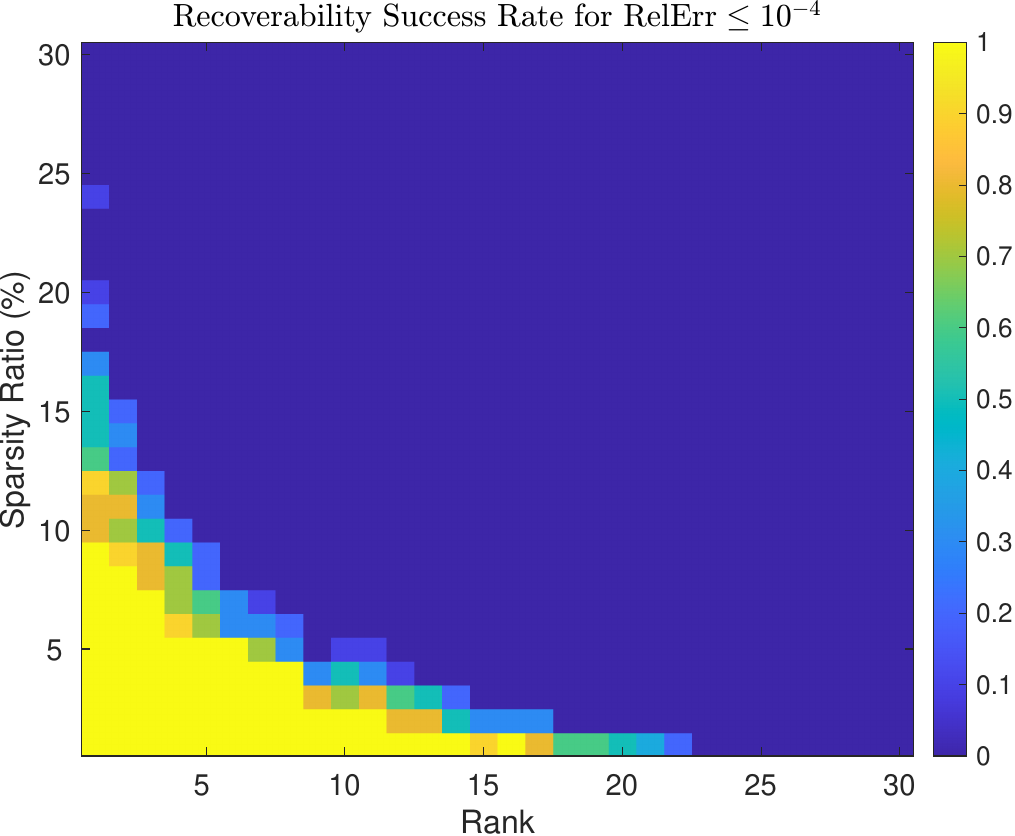} & \includegraphics[width=0.48\textwidth]{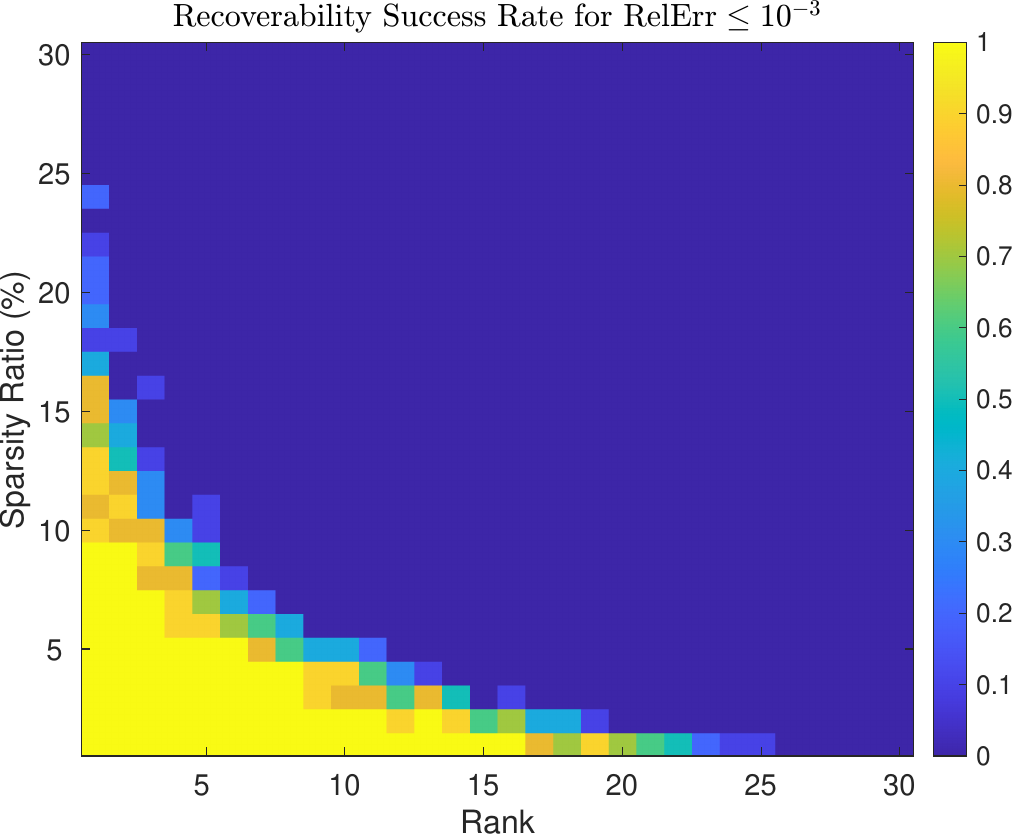}
  \end{tabular}
\caption{Recoverability results when $\mH$ is circulant (deterministic)}\label{fig:Hcirc}
\end{figure}

\subsection{Background removal and deblurring in surveillance videos}\label{sec:vid}

Let $\Vc=\Bc + \Oc$ represent the tensor of a surveillance video, where each pixel value has been normalized to lie within the interval $[0,1]$. 
Here, $\Bc$ corresponds to the static background, so $\mat(\Bc)$ is approximately rank 1, while $\Oc$ represents the moving objects. Since a pixel value of 1 is white and 0 is black, the extracted moving objects $\Oc$ are expected to have mostly white pixels. Thus, $1-\Oc$ is expected to be sparse. 
%We have $1-\mathcal{V}=-\mathcal{B}+1-\mathcal{O}.$
Let $\Mc = 1-\Vc, \Lc = -\Bc$, and $\Sc_0 =1-\Oc$. Then we have:
$$\Mc = \Lc+\Sc_0 \Longleftrightarrow \mM = \mL+\mS_0,$$
where $\mL=\mat(\Lc)$ is low-rank and $\mS_0=\mat(\Sc_0)$ is sparse.

Now suppose we obtain a blurred video $\mM_0 = \mH\mM=\mH\mL+\mH\mS_0:=\mL_0+\mH\mS_0$, which matches the setting in \eqref{equ:M0}. Applying Algorithm \ref{alg:gts_c} allows us to recover $\mS_0$, i.e., the \textbf{deblurred} moving objects. Under this framework, we simultaneously achieve background removal and deblurring. The algorithm also recovers $\mL_0=\mH\mL$, which corresponds to the blurred background video. However, this is typically not of primary interest in practical applications.

We use a video from the  BMC 2012 Background Models Challenge Dataset~\cite{vacavant2012benchmark}. Each frame is of size $240\times 320$ pixels, with a total of 300 frames. The video is blurred using a filter $\mH=\mG_2\times \mG_1,$ where the $\mG_i$'s are constructed as in \eqref{equ:Gi}, where
$$\mE_1= \begin{bmatrix}
0.4375&0.5625\\ 0.5625   & 0.4375
\end{bmatrix}, \quad 
\mE_2=
\begin{bmatrix}
 0.1123   & 0.3459   & 0.3446  &  0.1972\\
    0.1972  &  0.1123  &  0.3459 &   0.3446\\
    0.3446  &  0.1972  &  0.1123  &  0.3459\\
    0.3459  &  0.3446   & 0.1972 &   0.1123
    \end{bmatrix}.
    $$
Figure \ref{fig:vid}(a) shows the original video (10th frame), denoted by $\mM$, and Figure \ref{fig:vid}(b) displays the blurred version $\mM_0=\mH\mM$.

We then apply Algorithm \ref{alg:gts_c} with parameters $\rho_O=\rho_I=1, \epsilon_O = 10^{-7}, \epsilon_I=10^{-5}, T_O=100$, and $T_I = 10$. Figure \ref{fig:vid}(c) displays the recovered moving objects $S_0$ when LASSO is solved by FISTA, which are clearly deblurred. Figure \ref{fig:vid}(d) shows the recovered background when LASSO is solved by FISTA, which remains blurred as expected. The third row Figure \ref{fig:vid}(e)(f) is visually the same as the second row, but these are recovered when LASSO is solved by ADMM. Both methods took the maximum 100 iterations, but FISTA took only 266 seconds while ADMM took 704 seconds.

\begin{figure}[htb]
\begin{tabular}{cc}
(a) original (unknown)&(b) blurred (given $\Mc_0$)\\
\includegraphics[width=0.5\textwidth]{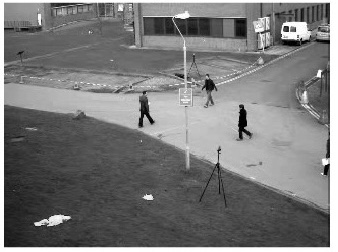}&
\includegraphics[width=0.5\textwidth]{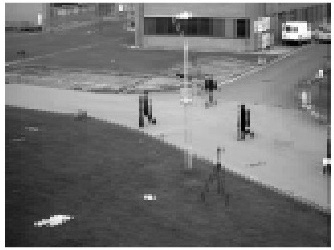}\\\\
(c) recovered moving objects - FISTA &(d) recovered background - FISTA\\
\includegraphics[width=0.5\textwidth]{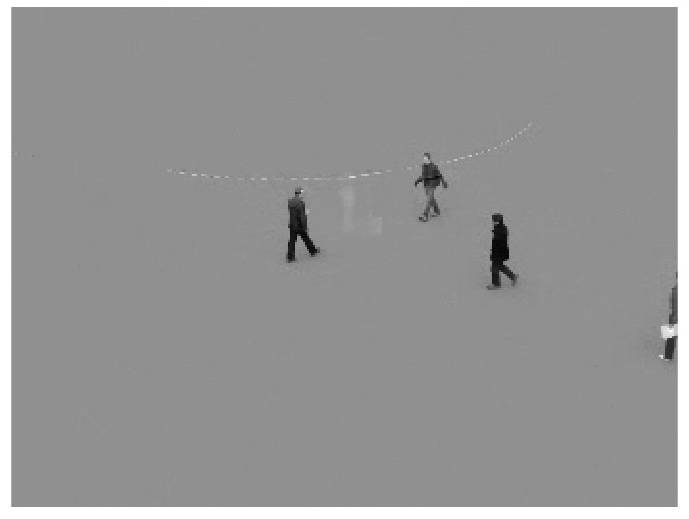}&
\includegraphics[width=0.5\textwidth]{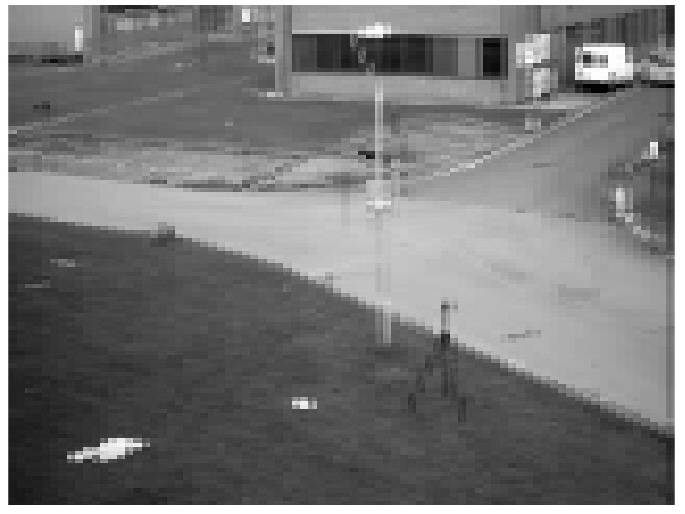}\\\\
(e) recovered moving objects - ADMM &(f) recovered background- ADMM\\
\includegraphics[width=0.5\textwidth]{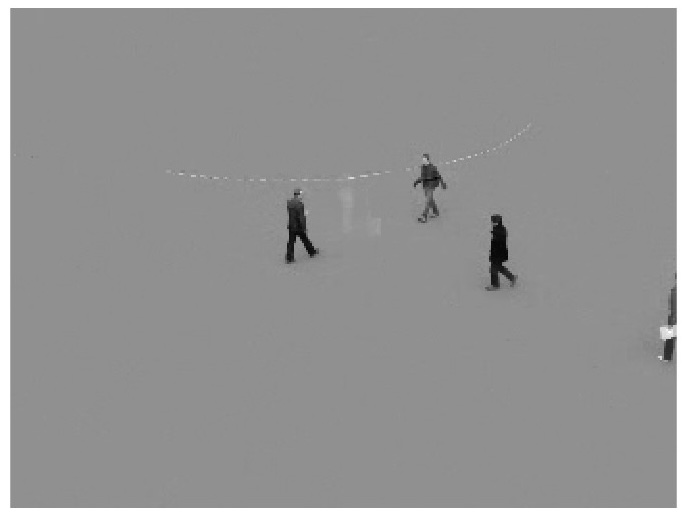}&
\includegraphics[width=0.5\textwidth]{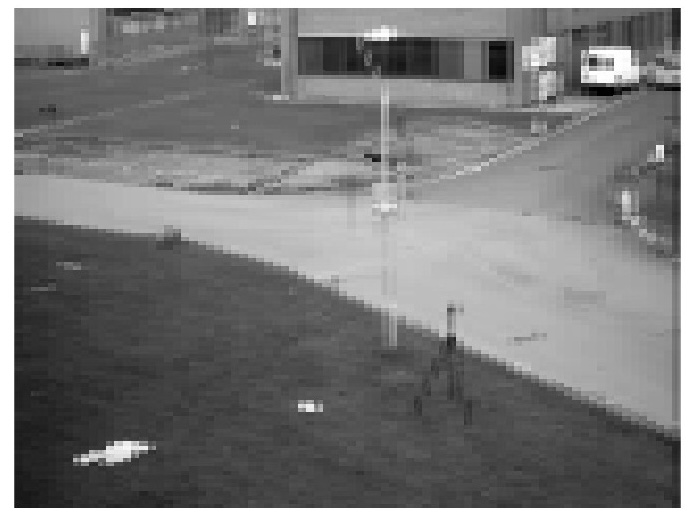}
\end{tabular}
\caption{Simultaneous background removal and deblurring}\label{fig:vid}
\end{figure}

%%%%%%%%%%%%%%%%%%%%%%%%%%%%%%%%%%%
\section{Discussion}
In this paper, we presented a comprehensive framework for solving the generalized matrix separation problem. Through ADMM and utilizing efficient algorithmic strategies for various structures of the filtering matrix $\mH$, we demonstrated significant improvements in both computational efficiency and recovery accuracy. We introduced a preconditioning method that improves convergence and is supported by theoretical guarantees under standard incoherence assumptions.
Based on this preconditioning technique, \textbf{we recommend Algorithm~\ref{alg:gms_c}  for the general matrix separation problem} and Algorithm~\ref{alg:gts_c}  for the specialized tensor case. Additionally, we provide parameter choices based on extensive numerical experiments.

Our numerical experiments confirm the robustness and scalability of the proposed algorithms across a variety of scenarios, including applications to video data for simultaneous background removal and deblurring. The ability of our methods to efficiently handle circulant, separable, and block-structured filters shows their practical relevance in real-world problems. 

The Singular Value Thresholding step that appears in every iteration of Algorithm~\ref{alg:gms_c} or Algorithm~\ref{alg:gts_c}, remains the primary computational bottleneck. Randomized SVD~\cite{rSVD} should be considered to speed up the computations for more realistic implementation. Moreover, 
%Future work may 
we may consider extensions to adaptive or time-varying filters and aim to establish stronger recovery guarantees under broader conditions.

\section*{Acknowledgments}

X. Chen is partially funded by NSF DMS-2307827. X. Chen would like to thank UNC
Wilmington’s Research Hub initiative.

\section*{Appendix}
\begin{proof}[Proof of Lemma \ref{lem}:]

Let $\diag(\mA) = [a_1,a_2 \cdots , a_n]^\top$ and $\diag(\mB) = [ b_1, b_2 \cdots , b_m]^\top$. 
%We have
%$$
%\diag(\mA \otimes \mB) = [ a_1b_1, \cdots, a_1b_m, \cdots, a_nb_1, \cdots, a_nb_m ] ^\top.
%$$
By definition of $\im(\cdot)$, 
\begin{align*}
\im_{m, n}(\diag(\mA \otimes \mB)) = \im_{m, n}\left(
\begin{bmatrix}
a_1\mB\\
&a_2\mB\\
&&\ddots\\
&&&a_n\mB\end{bmatrix}\right)=
\begin{bmatrix}
a_1b_1 & \cdots & a_nb_1 \\
\vdots & \ddots & \vdots \\
a_1b_m & \cdots & a_nb_m
\end{bmatrix},
\end{align*}
which is exactly the right hand side of \eqref{equ:lem}. 
\end{proof}

\subsection*{LASSO}
Given $\mA\in\R^{m\times n}, \vb\in\R^m, \lambda>0$, the Least Absolute Shrinkage and Selection Operator (LASSO)  solves the following optimization problem:
\begin{equation}\label{equ:lasso}
\min_{\vx\in\R^{n}} \frac{1}{2} \| \mA \vx - \vb \|_2^2 + \lambda \| \vx \|_1.
\end{equation}
The constant $\lambda > 0$ balances data fitting with regularization. If $\lambda$ is close to 0, then the solution $\vx$ will be similar to the ordinary least squares solution  \cite{HT09}. 

%%%%%%%%% Adding FISTA
LASSO has becoming an important tool in model selection and other sparsity constrained problems. One of the most popular methods for solving problem \eqref{equ:lasso} is in the class of iterative shrinkage-thresholding algorithms (ISTA) \cite{FN03, DDD04, FNW08}. It can be viewed as a proximal gradient method, which can be suitable for optimization problems consisting of a smooth term plus a non-smooth term. In this case, the smooth term is the quadratic $\frac{1}{2}\|\mA \vx - \vb\|_2^2$, and the non-smooth term is the $\ell_1$ regularizer $\lambda \|\vx\|_1$. 
%An alternative approach to solving \eqref{equ:lasso} is the proximal gradient method, which is designed for optimization problems consisting of a 
The method alternates between a gradient step on the smooth term, and a proximal step on the non-smooth term, which reduces to the soft-thresholding operator.

FISTA (Fast Iterative Shrinkage-Thresholding Algorithm) \cite{BATM2009} is an accelerated variant of ISTA, achieving a convergence rate of $O(1/k^2)$ (compared to $O(1/k)$ for ISTA) while keeping its simplicity. Let $f:\R^n\rightarrow\R$ be a continuously differentiable function with Lipschitz constant of the gradient $L$ such that
$$\|\nabla f(\vx)-\nabla f(\vy)\|_2\leq L\|\vx-\vy\|_2, \text{ for all }\vx, \vy\in\R^m.$$
In particular, for $f(\vx)=\frac{1}{2}\|\mA \vx - \vb\|_2^2$, the (smallest) Lipschitz constant of the gradient is $L=\|\mA\|^2$. FISTA can be applied to a general class of problems, but in the particular case of \eqref{equ:lasso}, a simple version (with constant stepsize in~\cite{BATM2009}) is the following
\begin{align*}
&\vx^{k+1}=\argmin_{\vx}\left\{\lambda\|\vx\|_1+\frac{L}{2}\|\vx-(\vy^k-\frac{1}{L}\nabla f(\vy^k))\|^2_2\right\},\\
&t^{k+1}=\frac{1+\sqrt{1+4(t^k)^2}}{2},\\
&\vy^{k+1}=  \vx^{k} + \frac{t_k - 1}{t_{k+1}}(\vx^{k+1} - \vx^k).
\end{align*}
Note that the $\vx$ update above can be further calculated as $\vx^{k+1}=S_{\lambda/L}(\vy^k-\frac{1}{L}\nabla f(\vy^k))=S_{\lambda/L}(\vy^k-\frac{1}{L}\mA^\top(\mA\vy^k-\vb))$.
Algorithm~\ref{alg:fista} presents the FISTA iteration for \eqref{equ:lasso}. Note that with the same $\mA$, we can solve for multiple right hand sides, resulting $\vb$ to be a matrix in Algorithm \ref{alg:fista}.

\begin{algorithm}[htb]
\caption{LASSO($\mA, \vb, \lambda; T, \epsilon$): FISTA \cite{BATM2009}}\label{alg:fista}
\begin{algorithmic}
\State \textbf{Input:} $\mA\in\R^{m\times n}, \vb\in\R^{m\times l}, \lambda >0$,  maximum number of iterations $T$ and the tolerance $\epsilon$.
\State \textbf{Output:} $\vx^{t}\in\R^{n\times l}$: an approximation of solution of \eqref{equ:lasso} at last iteration $t$
\State $L=\|\mA\|^2$
\State \textbf{Initialize:} $\vx^0, \vy^0\in \R^{n\times l}, t^0 = 1$
\For{$k=0,1,\cdots,T-1$}
    \State 1: $\vx^{k+1} = S_{\lambda/L}\!\left(\vy^k - \tfrac{1}{L}\mA^\top(\mA\vy^k - \vb)\right)$ 
    \State 2: $t^{k+1} = \frac{1 + \sqrt{1 + 4(t^k)^2}}{2}$
    \State 3: $\vy^{k+1} = \vx^{k} + \frac{t^k - 1}{t^{k+1}}(\vx^{k+1} - \vx^k)$
    \State Terminate if $\|\vx^{k+1}-\vx^k\|_F/(\|\vx^k\|_F+1) < \epsilon$ 
\EndFor
\end{algorithmic}
\end{algorithm}

%%%%%%%%%%%%%%%%%%%%%%%%%%%%%

We can also use ADMM to solve LASSO problems. We recast \eqref{equ:lasso} as
\[
(\hat \vx, \hat \vz) = \argmin_{\vx,\vz} \frac{1}{2} \| \mA \vx - \vb \|_2^2 + \lambda \| \vz \|_1, \qquad \text{subject to }\vx-\vz=0,
\]
which is a special case of ADMM. Then, \eqref{equ:admm1}--\eqref{equ:admm3} become
\begin{align}\label{equ:l1}
\vx^{k+1} &= (\mA^\top \mA + \rho \mI)^{-1}(\mA^\top \vb + \rho (\vz^k - \vu^k))\\
\vz^{k+1} &= S_{\lambda/\rho}(\vx^{k+1} + \vu^k)\\\label{equ:l3}
\vu^{k+1} &= \vu^k + \vx^{k+1} - \vz^{k+1}.
\end{align}

The first step \eqref{equ:l1} is to solve the linear system:
\begin{align}\label{lasso:step1}
(\mA^\top \mA + \rho \mI)\vx^{k+1} = \mA^\top \vb + \rho (\vz^k - \vu^k).
\end{align}

Since the same coefficient matrix, $\mA^\top \mA + \rho \mI$, is used repeatedly, we gain efficiency by pre-factorizing it.

For example, we may precompute an eigendecomposition of $\mA^\top \mA$ as 
\[
\mA^\top \mA = \mV \mathbf{\Sigma} \mV^\top,
\]
where $\mV$ is orthogonal and $\mathbf{\Sigma}$ is diagonal. Then
\[
\mA^\top \mA + \rho \mI = \mV (\mathbf{\Sigma} + \rho \mI) \mV^\top,
\]
and the solution to $(\mA^\top \mA + \rho \mI)\vx = \vr$ is
\[
\vx = \mV (\mathbf{\Sigma} + \rho \mI)^{-1}\mV^\top \vr.
\]

In practice, rather than forming $\mA^\top \mA$, it is preferable from a numerical perspective to compute a singular value decomposition (SVD) of $\mA$ (or a truncated SVD when $\mA$ is large). This avoids squaring the condition number and is typically more numerically stable.
%This is efficient to compute since $\mathbf{\Sigma} + \rho \mI$ is diagonal. 
This leads to Algorithm~\ref{alg:lasso2} following \eqref{equ:l1}--\eqref{equ:l3}.

\begin{algorithm}[htb]
\caption{LASSO($\mA, \vb, \lambda; \rho, T, \epsilon$): ADMM}\label{alg:lasso2}
\begin{algorithmic}
\State\textbf{Input:} $\mA\in\R^{m\times n}, \vb\in\R^{m\times l}, \lambda >0, \rho>0$, maximum number of iterations $T$ and the tolerance $\epsilon$.
\State\textbf{Output:} $\vx^{t}\in\R^{n\times l}$: an approximation of solution of \eqref{equ:lasso} at last iteration $t$
\State\textbf{Initialize:} $\vz^0, \vu^0 \in\R^{n\times l}$
\State Compute an eigendecomposition $\mA^\top \mA = \mV \mathbf{\Sigma} \mV^\top$
\State $Atb = \mA^\top \vb$ %\hfill{$O(mnl)$}
\For{$k=0,1,\cdots,T-1$} 
\State 1: Solve $(\mathbf{\Sigma} + \rho \mI)\vy = \mV^\top(Atb + \rho (\vz^k - \vu^k))$, then set $\vx^{k+1} = \mV \vy$
\State 2: $\vz^{k+1} = S_{\lambda/\rho}(\vx^{k+1} + \vu^k)$ %\hfill{$O(ml)$}
\State 3: $\vu^{k+1} = \vu^k + \vx^{k+1} - \vz^{k+1}$ %\hfill{$O(ml)$}
\State Terminate if $\|\vx^{k+1}-\vx^k\|_F/(\|\vx^k\|_F+1) <\epsilon$ 
\EndFor
\end{algorithmic}
\end{algorithm}

%\begin{remark}
%Traditionally, $\vb\in\R^{m\times1}$ is a  vector, but Algorithms \ref{alg:lasso2}, \ref{alg:lasso-chol}, and \ref{alg:lasso_c}, will work exactly the same if $\vb\in\R^{m\times l}$. The output has the same dimension as $\vb$.
%\end{remark}
Comparing Algorithm \ref{alg:fista} with Algorithm \ref{alg:lasso2}, we see that FISTA is very simple to implement and it does not require parameter tuning.

\bibliographystyle{amsplain}
\bibliography{ref_alg}

\end{document}